\RequirePackage{fix-cm}
\documentclass[smallextended]{svjour3}       % onecolumn (second format)
\smartqed  % flush right qed marks, e.g. at end of proof
\usepackage{graphicx}
\usepackage{subfigure}
\usepackage{enumerate}
\usepackage{xcolor}
\usepackage{amsmath}
\usepackage{cases}%给大括号后多个公式分别编号，需放在宏包{amsmath}后面
\usepackage{amssymb,epsfig,multirow}
\usepackage{float}
\usepackage{color}
\usepackage{booktabs}
\usepackage{enumerate}
\usepackage{epstopdf}
\usepackage{algorithmicx,algorithm}
\usepackage{makecell}
\usepackage{url}%bib引用网页

\newtheorem{thm}{Theorem}%[section]
\newtheorem{lem}{Lemma}%[section]
\newtheorem{cor}{Corollary}%[section]
\newtheorem{rem}{Remark}%[section]
\newtheorem{sche}{Scheme}%[section]
\newtheorem{algo}{Algorithm}%[section]
\newtheorem{defi}{Definition}%[section]

\begin{document}
\title{
Alternating direction method of multipliers for convex programming: a lift-and-permute scheme
%Gauss-Seidel alternating direction method for convex programming: lift and permute
	\thanks{This research was supported  by
Beijing Natural Science Foundation Z180005 and
National Natural Science Foundation of China
under grants 12171021 and 11822103.}
}

\titlerunning{A lift-and-permute ADMM scheme}        % if too long for running head

\author{ Shiru Li  \and Yong Xia  \and   Tao Zhang }


\institute{
	Shiru Li  \and Yong Xia  \and   Tao Zhang  \at
School of Mathematical Sciences, Beihang University, Beijing, 100191, P. R. China
	\email{lishiru@buaa.edu.cn  (S.R. Li); yxia@buaa.edu.cn (Y. Xia); shuxuekuangwu@buaa.edu.cn (T. Zhang, corresponding author)}
}

\date{Received: date / Accepted: date}

\maketitle

\begin{abstract}
A lift-and-permute scheme of  alternating direction method of multipliers (ADMM) is proposed for linearly constrained convex programming. It contains not only the newly developed balanced augmented Lagrangian method and its dual-primal variation, but also the proximal ADMM and Douglas-Rachford splitting algorithm. It helps to propose accelerated algorithms with worst-case $O(1/k^2)$ convergence rates in the case that the objective function to be minimized is strongly convex.
	%\subclass{65F20, 90C26, 90C32, 90C20}
\end{abstract}
\section{Introduction}
Consider  the convex programming problem with linear equality constraints:
\begin{eqnarray*}
({\rm P})~~\min_x &&f(x)\\
{\rm s.t.}&&Ax=b,
\end{eqnarray*}
where $f:\mathbb{R}^n\rightarrow\mathbb{R}$ is closed, proper, convex, but not necessarily smooth, $A\in\mathbb{R}^{m\times n}$, and $b\in\mathbb{R}^{m}$.

As a fundamental and efficient approach for solving (P),
the classical augmented Lagrangian method (ALM), dating back to  \cite{1969Multiplier,1969A},  reads as
\begin{eqnarray*}
~
\left\{
\begin{array}{lcl}
x^{k+1}&\in&\arg\min\limits_{x}\{f(x)+(\lambda^k)^T(Ax-b)+\frac{\beta}{2}\|Ax-b\|^2\},\\
\lambda^{k+1}&=&\lambda^{k}+\beta(Ax^{k+1}-b),
\end{array}
\right.%\label{ALM}
\end{eqnarray*}
where $\lambda$ is Lagrange multiplier corresponding to the equality constraints and $\beta>0$ is a fixed penalty parameter.
The update of $x^{k+1}$ may have no closed-form solution due to the coupling of $\|Ax-b\|^2$ and $f(x)$. The following  proximal ALM \cite{rockafellar1976augmented} tries to overcome
this difficulty by introducing a carefully designed proximity term:
\begin{eqnarray*}
\left\{
\begin{array}{rcl}
x^{k+1}&\in&\arg\min\limits_{x}\{f(x)+(\lambda^k)^T(Ax-b)
+\frac{\beta}{2}\|Ax-b\|^2+\frac{1}{2}\|x-x^k\|_G^2\},\\
\lambda^{k+1}&=&\lambda^{k}+\beta(Ax^{k+1}-b).
\end{array}
\right.
\end{eqnarray*}
In fact, by choosing $G=rI_n-\beta A^TA$ with   $r>\beta\rho(A^TA)$ and  $\rho(\cdot)$ being the spectral norm, we obtain a reduced update of $x^{k+1}$,
%One can easily verify that now the update of $x^{k+1}$ is reduced to
\begin{eqnarray*}
x^{k+1}&\in&\arg\min\limits_{x}\left\{f(x)+\frac{r}{2}\left\|x-x^k+\frac{1}{r}A^T
[\lambda^k+\beta(Ax^k-b)]\right\|^2\right\},
%\label{LALM-x}
\end{eqnarray*}
which is easy to solve if the following proximal mapping of $f(x)$
$$
\text{prox}_{\gamma f}(x):=(I+\gamma\partial f)^{-1}(x):= \arg\min\limits_{y}f(y)+\frac{1}{2\gamma}\|y-x\|^2
$$
has a closed-form (or easy-to-compute) solution.

  Since $r$ should be set larger than a fixed  proportion of $\rho(A^TA)$,
the shortcoming of the above proximal ALM is that, for large $\rho(A^TA)$,
the iteration sequence $\{x^{k+1}\}$ will get stuck in updating.
There is an alternative first-order primal-dual method presented in \cite{chambolle2011first} with the following iteration formula:
\[%\begin{eqnarray}\label{cp}
	\begin{cases}			x^{k+1}\in\arg\min\limits_{x}\{f(x)+\frac{r}{2}
\|x-(x^k-\frac{1}{r}A^T\lambda^k)\|^2\},\\
\lambda^{k+1}=\lambda^{k}+\frac{1}{s}[A(2x^{k+1}-x^k)-b],
	\end{cases}
\]%\end{eqnarray}
where $r>0$ and $s>0$ satisfy that
$
rs>\rho(A^TA).
$
Again, for large $\rho(A^TA)$, either $r$ or $s$ must be large enough. Then either $\|x^{k+1}-x^k\|$ or $\|\lambda^{k+1}-\lambda^{k}\|$ is small. Recently, He et al. \cite{hmy} relaxed the requirement  to $rs>0.75\rho(A^TA)$.

In order to completely remove the restriction on the step-sizes $r$ and $s$,
He and Yuan \cite{he2021balanced} proposed a simple but effective augmented Lagrangian method, the so-called  balanced ALM,
%named the Balanced augmented Lagrangian method (Balanced ALM) by balancing the difficulties of subproblems without adding additional conditions. In fact, Balanced ALM can also be applied to more general convex programming problems with linear equality or inequality constraints.
which reads as
\begin{eqnarray}\label{BALM}
	\begin{cases}			
x^{k+1}=\arg\min\limits_{x}
\{f(x)+\frac{r}{2}\|x-(x^k-\frac{1}{r}A^T\lambda^k)\|^2\},\\
\lambda^{k+1}=\lambda^{k}+(\frac{1}{r}AA^T+\delta I_m)^{-1}[A(2x^{k+1}-x^k)-b],
	\end{cases}
\end{eqnarray}
where $r > 0$ and $\delta > 0$ are arbitrary parameters. Different from the classical ALM and its proximal variations, balanced ALM \eqref{BALM} has an additional cost in  updating $\lambda^{k+1}$ by solving the following linear equations
\begin{equation}
\left(\frac{1}{r}AA^T+\delta I_m\right)(\lambda-\lambda^k)-(A[2x^{k+1}-x^k]-b)=0.\label{BALM-equ}
\end{equation}
%the Balanced ALM eases the difficulty of solving the $x$-subproblem, only increases the amount of calculation to solve linear equations when updating the multiplier. From this perspective, the method balances the difficulty of solving the two subproblems. Compared with linearized ALM, the Balanced ALM relax the step size condition.
Following this idea, Xu \cite{xu2021dual} proposed a dual-primal balanced ALM with the same complexity per iteration as balanced ALM. The iteration is given by
%which has the same computational difficulty as the balanced ALM.
\begin{eqnarray}\label{311}
\begin{cases}
x^{k+1}=\arg \min\limits_{x} \left\{f(x)+\frac{r}{2}\left\|x-\left\{x^{k}-\frac{1}{r} A^{T}\left(2 \lambda^{k}-\lambda^{k-1}\right)\right\}\right\|^{2}\right\},\\\lambda^{k+1}=\lambda^{k}+\left(\frac{1}{r} A A^{T}+\delta I_{m}\right)^{-1}\left(A x^{k+1}-b\right),
\end{cases}
\end{eqnarray}
where $r>0$ and $\delta>0$ are arbitrary parameters.

%It can be proved that our method is equivalent to the balanced augmented Lagrangian method when solving problem (P).
 (P) can be regarded as a special case of the following two-block  problem:
\begin{eqnarray}\label{P2}
%\begin{array}{cl}
\min\limits_{x,y} \{f(x)+g(y) :~Ax+By=b\},
%\end{array}
\end{eqnarray}
where $A\in\mathbb{R}^{m\times n}$, $B\in\mathbb{R}^{m\times l}$,  $b\in\mathbb{R}^{m}$, $f:\mathbb{R}^{n}\rightarrow\mathbb{R}$, $g:\mathbb{R}^{l}\rightarrow\mathbb{R}$ are proper and convex, and $g$ is additionally closed.
%Problem (P) can be regarded as a special case of problem \eqref{P2} when $g(y)=0$ and $B=0$.
The alternating direction method of multipliers (ADMM) \cite{gabay1976dual,glowinski1975appro} is popular for solving \eqref{P2}.
The iterative formula reads as
\begin{eqnarray*}%\label{ADMM-2}
\begin{cases}
x^{k+1}\in\arg\min\limits_{x}\mathcal{L}_\beta\left(x, y^k,\lambda^k\right),\\
y^{k+1}\in\arg\min\limits_{y}\mathcal{L}_\beta\left(x^{k+1}, y,\lambda^k\right),\\
\lambda^{k+1}=\lambda^k+\beta(Ax^{k+1}+By^{k+1}-b),
\end{cases}
\end{eqnarray*}
where $\mathcal{L}_\beta(x,y,\lambda)$ is the augmented Lagrangian function of  \eqref{P2} defined as
\begin{eqnarray*}
\mathcal{L}_\beta(x,y,\lambda)=f(x)+g(y)+\lambda^T(Ax+By-b)+\frac{\beta}{2}\|Ax+By-b\|^2.
\end{eqnarray*}

Similar to the idea of proximal ALM, proximal ADMM \cite{eckstein1994some} decouples $\|Ax+By-b\|^2$ and the objective function by introducing suitable proximity terms.
The iterative formula can be written as
\begin{eqnarray*}%\label{L-ADMM}
\begin{cases}
x^{k+1}\in\arg\min\limits_{x}\{\mathcal{L}_\beta\left(x, y^k,\lambda^k\right)+\frac{1}{2}\|x-x^k\|_C^2\},\\
y^{k+1}\in\arg\min\limits_{y}\{\mathcal{L}_\beta\left(x^{k+1}, y,\lambda^k\right)+\frac{1}{2}\|x-x^k\|_D^2\},\\
\lambda^{k+1}=\lambda^k+\beta(Ax^{k+1}+By^{k+1}-b),
\end{cases}
\end{eqnarray*}
where $C$ and $D$ are positive semidefinite matrices. As shown above, letting  $C=rI_n-\beta A^TA$ could lead to an easy-to-solve $x$-subproblem.

%Another way to solve
Problem \eqref{P2} can be alternatively solved by Douglas-Rachford splitting (DRS) algorithm \cite{2005Splitting}. In general, DRS finds the zero point of the sum of two maximal monotone operators $\mathcal{A}$ and $\mathcal{B}$ via the following iterative formula:
\begin{eqnarray*}
\begin{cases}
w^{k+1}=(I+\tau\mathcal{A})^{-1}(2z^k-w^k)+w^k-z^k,\\
z^{k+1}=(I+\tau\mathcal{B})^{-1}(w^{k+1}).
\end{cases}
\end{eqnarray*}
The equivalence between ADMM and DRS has been established in  \cite{2005Splitting}.

In this paper,
we propose a lift-and-permute scheme of ADMM  for solving problem (P). Each algorithm in our scheme  employs a variant ADMM in a permuted order of updating variable to solve the same dual problem of (P) with additional copying variables.
Surprisingly,  we can show that the above mentioned balanced ALM and its dual-primal variation,   proximal ADMM and DRS algorithm all correspond to algorithms in our scheme. With the help of this understanding,  we propose in the first time an accelerated balanced ALM and its dual-primal variation with the worst-case $O(1/k^2)$ convergence rate in the case that $f(x)$ is strongly convex.

The remainder is organized as follows. Section 2 presents the motivation and the equivalence between DRS and (dual-primal) balanced ALM.
Section 3 proposes a lift-and-permute scheme of ADMM and show it includes (dual-primal) balanced ALM and proximal ADMM. Section 4 presents accelerated  algorithms and convergence rate analysis. Conclusion and future works are given in Section 5.

\section{Why lift}

\subsection{Solve a lifted version via ADMM}
First, we can see that  balanced ALM \eqref{BALM} and its dual-primal variation  \eqref{311} have the same computational complexity as employing the classical ADMM to solve  the following reformulation of (P) with additional copied variables:
\begin{eqnarray*}
\min\limits_{x,y}\{f(x):~Ay=b,~x=y\},
\end{eqnarray*}
which is a case of \eqref{P2}.
The corresponding  ADMM is rewritten as
\begin{eqnarray}
\left\{
\begin{array}{rcl}
x^{k+1}&=&\arg\min\limits_{x}\{f(x)+\frac{\beta}{2}\|x-(y^k-\frac{\lambda_2^k}{\beta})\|\},\\
y^{k+1}&=&\arg\min\limits_{y}
\{ (\lambda_1^k )^TAy- (\lambda_2^k )^Ty+\frac{\beta}{2}\|Ay-b\|^2+\frac{\beta}{2}\|x^{k+1}-y\|^2\},\\
\lambda_1^{k+1}&=&\lambda_1^{k}+\beta(Ay^{k+1}-b),\\
\lambda_2^{k+1}&=&\lambda_2^{k}+\beta(x^{k+1}-y^{k+1}).\\
\end{array}
\right.\label{ADMM2}
\end{eqnarray}
Note that the update of $x^{k+1}$
amounts to the proximal solvers  $\text{prox}_{\frac{1}{\beta}f}(y^k-\frac{\lambda_2^k}{\beta})$. The update of $y^{k+1}$ is equivalent to solving the linear equation system
\[
\left( A^TA+ I\right)z= q^k:=\frac{1}{\beta}\left(\lambda_2^k+\beta A^Tb+\beta x^{k+1}-A^T\lambda_1^k\right).
\]
When $m\ll n$,  by Sherman-Morrison formula, we have
\[
z=\left( A^TA+ I\right)^{-1}q^k=\left( I-A^T\left(I+AA^T\right)^{-1}A\right)q^k
=q^k-A^T\widetilde{y},
\]
where $\widetilde{y}$ is a solution of  the $m$-dimensional system of equations
\begin{equation*}
\left(I+AA^T\right)\widetilde{y}=Aq^k.
\end{equation*}
%and return
%\[
%y=q^k-A^T\widetilde{y}.
%\]
That is, the main computational complexity is the same as \eqref{BALM-equ}.
Therefore, we have observed that
all iterative subproblems in \eqref{ADMM2} are relatively easy to solve. Moreover,
according to the convergence theory of the classical ADMM, there is no need to assume any restrictive conditions on $\beta$.

\subsection{Equivalence between  DRS and (dual-primal) balanced ALM  }

%Comparing algorithm \eqref{cp} and balanced ALM \eqref{BALM}, balanced ALM \eqref{BALM} needs additional calculating $(\frac{1}{r}AA^T+\delta I)^{-1}$. If we  reformulate problem (P) as:
%\begin{eqnarray*}
%\min\limits_{x}\{f(x)+\phi_{Ax=b}(x)\},
%\end{eqnarray*}
%where $\phi_{Ax=b}(x)$ is indicator function of $Ax=b$. The proximal mapping of $\phi_{Ax=b}(x)$ with $AA^T$ being invertible is given by:
%$$ \arg\min\limits_{y}\phi_{Ay=b}(y)+\frac{1}{2}\|y-x\|^2=x-A^T(AA^T)^{-1}(Ax-b).$$
% If $AA^T$ is not invertible,
For a  parameter $\sigma>0$, we lift  (P) as
\[ %\begin{eqnarray}\label{dr1}
 \min\limits_{x,y} \{f(x):~ Ax=\sigma y,~\sigma y=b\}
\] %\end{eqnarray}
and then  reformulate it as
 \begin{eqnarray}\label{dr2}
 \min\limits_{z=(x,y)} \{f(x)+\phi_{\sigma y=b}(y)+\varphi_{Ax=\sigma y}(z)\},
 \end{eqnarray}
 where $\phi_{\sigma y=b}(y)$ and $\varphi_{Ax=\sigma y}(z)$ are indicator functions of $\sigma y=b$ and $Ax=\sigma y$, respectively. The proximal mappings of $\phi_{\sigma y=b}(y)$  and $\varphi _{Ax=\sigma y}$ are given by
\begin{eqnarray*}
&&\arg\min\limits_{y'}\{\phi_{\sigma y=b}(y')+\frac{1}{2}\|y-y'\|^2\}=b/\sigma,\\
&&\arg\min\limits_{z'}\{\varphi_{Ax=\sigma y}(z')+\frac{1}{2}\|z-z'\|^2\}=z+
\begin{pmatrix}
 A^T\\-\sigma I
\end{pmatrix}
 (AA^T+\sigma^2I)^{-1}(
 A~-\sigma I
)z.
 \end{eqnarray*}
One can easily see that  applying DRS algorithm for solving \eqref{dr2} has the same computational complexity as that of balanced ALM \eqref{BALM} and its dual-primal variation \eqref{311} for solving (P).  Furthermore, we can establish the equivalence between DRS and balanced ALM \eqref{BALM} and then extend the equivalence to dual-primal balanced ALM \eqref{311}.
 \begin{thm}Let $F(z)=f(x)+\phi_{\sigma y=b}(y)$.
 Balanced ALM \eqref{BALM} for solving  (P) is equivalent to applying DRS  with $\mathcal{A}=\partial\varphi_{Ax=\sigma y}$ and $\mathcal{B}=\partial F$ to solve  \eqref{dr2} under the special parametric settings $ \tau r=1$ and $\tau\sigma^2=\delta$.
 \end{thm}
 \begin{proof}
 	Applying DRS with $\mathcal{A}=\partial\varphi_{Ax=\sigma y}$ and $\mathcal{B}=\partial F$  yields
 	\begin{eqnarray*}
 		\begin{cases}
 			w^{k+1}=\arg\min\limits_{z}\{ \varphi_{Ax=\sigma y}(z)+\frac{1}{2\tau}\|z-(2z^k-w^k)\|^2\}+w^k-z^k,\\
 			z^{k+1}=\arg\min\limits_{z} \{F(z)+\frac{1}{2\tau}\|z-w^{k+1}\|^2\}.
 		\end{cases}
 	\end{eqnarray*}
 	Let $ w^k=(
 	\widetilde{x}^k,~\widetilde{y}^k
)$, $z^k=(
 	x^k,~y^k
 	)$ and $H=(AA^T+\sigma^2I)^{-1}$. According to the optimality conditions, we can explicitly rewrite the above iterative formula as
 	\begin{eqnarray}
 	&&\widetilde{x}^{k+1}=-A^THA(2x^k-\widetilde{x}^k)+\sigma A^TH(2y^k-\widetilde{y}^k)+x^k\nonumber,\\
 	&&\widetilde{y}^{k+1}=\sigma HA(2x^k-\widetilde{x}^k)-  \sigma^2 H(2y^k-\widetilde{y}^k)+y^k\label{dr5},\\
 	&&	0\in \partial f({x}^{k+1})+\frac{1}{\tau}(x^{k+1}-\widetilde{x}^{k+1})\label{dr3},\\
 	&&	\sigma y^{k+1}=b\label{dr6}.
 	\end{eqnarray}
 	According to \eqref{dr5} and \eqref{dr6}, we  obtain
 	\begin{eqnarray}\label{dr7}
 \widetilde{x}^{k+1}=\frac{1}{\sigma}(-A^T\widetilde{y}^{k+1}+\frac{1}{\sigma}A^Tb)+x^{k}.
 	\end{eqnarray}
Substituting \eqref{dr7}	into \eqref{dr3} yields that
  \[
 	0\in \partial f({x}^{k+1})+\frac{1}{\tau}(x^{k+1}-x^k)+\frac{1}{\tau \sigma}A^T(\widetilde{y}^{k+1}-\frac{1}{\sigma}b).
 	\]
By first rewriting \eqref{dr7} as an update of $\widetilde{x}^k$ from $\widetilde{y}^k$ and $x^{k-1}$ and then
taking  it into \eqref{dr5}, we have
 	\begin{eqnarray}
 	\widetilde{y}^{k+1}%=&&\sigma HA(2x^k-\widetilde{x}^k)-\sigma^2 H(2y^k-\widetilde{y}^k)+y^k\nonumber\\
 =&&\sigma HA(2x^k-x^{k-1})+H(AA^T+\sigma^2 I)\widetilde{y}^k\nonumber\\&&-\frac{1}{\sigma}H(AA^T+\sigma^2I)b-\sigma Hb+\frac{1}{\sigma}b\nonumber\\=&&\widetilde{y}^k+
 	\sigma H[A(2x^k-x^{k-1})-b].\nonumber
 	\end{eqnarray}
Let $x^{k+1}:=x^{k+1}$, $\lambda^{k}:=\frac{1}{\tau\sigma}(\widetilde{y}^{k+1}-\frac{1}{\sigma}b)$, $ 1/\tau =r$ and $\tau\sigma^2=\delta$. Then we can recover balanced ALM \eqref{BALM}.
 \end{proof}
 \begin{thm}Let $F(z)=f(x)+\phi_{\sigma y=b}(y)$.
Dual-primal balanced ALM \eqref{311} for solving  (P) is equivalent to applying DRS  with $\mathcal{A}=\partial F$ and $\mathcal{B}=\partial\varphi_{Ax=\sigma y}$ to solve  \eqref{dr2} under the special parametric settings $ \tau  r=1$ and $\tau\sigma^2=\delta$.
 \end{thm}
 \begin{proof}
 	Applying DRS with $\mathcal{B}=\partial\varphi_{Ax=\sigma y}$ and $\mathcal{A}=\partial F$ yields that
 	\begin{eqnarray*}
 		\begin{cases}
 			w^{k+1}=\arg\min\limits_{z}\{F(z) +\frac{1}{2\tau}\|z-(2z^k-w^k)\|^2\}+w^k-z^k,\\
 			z^{k+1}=\arg\min\limits_{z} \{\varphi_{Ax=\sigma y}(z)+\frac{1}{2\tau}\|z-w^{k+1}\|^2\}.
 		\end{cases}
 	\end{eqnarray*}
 	Let $ w^k=(
 	\widetilde{x}^k,~\widetilde{y}^k
 	)$, $z^k=(
 	x^k,~y^k
 	)$, $\widehat{x}^{k+1}=\widetilde{x}^{k+1}-\widetilde{x}^k+x^k$ and $H=(AA^T+\sigma^2I)^{-1}$. It follows from the optimality conditions that
 	\begin{eqnarray}
 	&&0\in\partial f(\widehat{x}^{k+1})+\frac{1}{\tau}(\widehat{x}^{k+1}-2x^k+\widetilde{x}^k),\label{dr9}\\
 	&&\widetilde{y}^{k+1}-\widetilde{y}^k+y^k=\frac{1}{\sigma}b,\label{dr10}\\
 	&&x^{k+1}=-A^THA\widetilde{x}^{k+1}+\sigma A^TH\widetilde{y}^{k+1}+\widetilde{x}^{k+1},\label{dr11}\\
 	&&y^{k+1}=\sigma HA\widetilde{x}^{k+1}-\sigma^2 H\widetilde{y}^{k+1}+\widetilde{y}^{k+1}.\label{dr12}
 	\end{eqnarray}
By combining \eqref{dr11} with \eqref{dr12}, we observe that
 	\begin{eqnarray}
 x^{k+1}=-\frac{1}{\sigma}A^T(y^{k+1}-\widetilde{y}^{k+1})+\widetilde{x}^{k+1}.\label{dr13}
 	\end{eqnarray}
According to the definition of $\widehat{x}^{k+1}$, we have
 	\begin{eqnarray}
 	A^T  \widetilde{y}^{k+1}+\sigma(\widetilde{x}^{k+1}-\widehat{x}^{k+1}) \overset{\eqref{dr13}}{=}
 A^T(\widetilde{y}^{k+1}-\widetilde{y}^k+y^k)
 \overset{\eqref{dr10}}{=}\frac{1}{\sigma}A^Tb.\label{dr14}
 	\end{eqnarray}
 	We also have
 \begin{eqnarray*}		\frac{1}{\sigma}b+\widetilde{y}^{k+1}-\widetilde{y}^{k+2}& \overset{\eqref{dr10}}{=}&y^{k+1} \overset{\eqref{dr12}}{=}\sigma HA\widetilde{x}^{k+1}-\sigma^2 H\widetilde{y}^{k+1}+\widetilde{y}^{k+1}\nonumber\\&\overset{\eqref{dr14}}{=}&-H(AA^T+\sigma^2I)\widetilde{y}^{k+1}+\widetilde{y}^{k+1}+\sigma HA\widehat{x}^{k+1}+\frac{1}{\sigma}HAA^Tb\nonumber\\&=&\sigma HA\widehat{x}^{k+1}+\frac{1}{\sigma}HAA^Tb,
 	\end{eqnarray*}
or equivalently,
\begin{eqnarray}\label{dr15}
 	\widetilde{y}^{k+2}=\widetilde{y}^{k+1}-\sigma H(A\widehat{x}^{k+1}-b
 	).\end{eqnarray}
 	Furthermore, we can deduce
 	\begin{eqnarray}\label{dr16}
 -2x^k+\widetilde{x}^k&\overset{\eqref{dr13}}{=}&\frac{2}{\sigma}A^T(y^k-\widetilde{y}^k)-\widetilde{x}^k\nonumber\\&\overset{\eqref{dr14}}
 {=}&\frac{2}{\sigma}A^Ty^k-\frac{1}{\sigma}A^T\widetilde{y}^k-\widehat{x}^k-\frac{1}{\sigma^2}A^Tb\nonumber\\&\overset{\eqref{dr10}}{=}&A^T(-\frac{2}{\sigma}\widetilde{y}^{k+1}+\frac{1}{\sigma}\widetilde{y}^{k}+\frac{1}{\sigma^2}b)-\widehat{x}^k.
 	\end{eqnarray}
By substituting \eqref{dr16} into \eqref{dr9} and then combining it with \eqref{dr15}, 	we recover  dual-primal balanced ALM \eqref{311}  with $x^{k+1}:=\widehat{x}^{k+1}$, $\lambda^k:=-\frac{1}{\tau\sigma}(\widetilde{y}^{k+1}-\frac{1}{\sigma}b)$, $1/\tau=r$ and $\tau\sigma^2=\delta$.
 \end{proof}

\section{A lift-and-permute scheme of ADMM}

%We rearrange the order of updates of Gauss-Seidel ALM in solving the dual problem of problem (P) that additionally introduces equality constraints. Adding additional equality constraints to the dual problem leads to an increase in the original variables. Therefor our method called a lift-and-permute scheme of Gauss-Seidel ADM.
Different from the primal lift as in Section 2, we  lift the dual problem of (P),
\begin{eqnarray}\label{dual}
\max\limits_\lambda\min\limits_x\{f(x)+\lambda^T(Ax-b)\}
=\max\limits_\lambda\{-f^*(-A^T\lambda)-\lambda^Tb\},
\end{eqnarray}
%By introducing an equality constraint to \eqref{dual},  \eqref{dual} is equivalent to the following problem:
%According to the analysis in section 2.1, we can lift negative \eqref{dual} to the
to the following reformulation:
\begin{eqnarray}\label{dual1}
\begin{split}
-\min\limits_{u,v,\lambda} \{f^*(u)+\lambda^Tb:~ -A^Tv=u,~v=\lambda\}.
\end{split}
\end{eqnarray}
The augmented Lagrangian function of problem \eqref{dual1} is given by
\begin{eqnarray*}
&&\mathcal L_{\beta_1,\beta_2}(u,v,\lambda,\bar x,\bar{y})\\
&&=f^*(u)+\lambda^Tb+\bar x^T(u+A^Tv)+\bar{y}^T(v-\lambda)+\frac{\beta_1}{2}\|u+A^Tv\|^2+\frac{\beta_2}{2}\|v-\lambda\|^2,
\end{eqnarray*}
where $\bar{x},\bar{y}$ are Lagrange multipliers, $\beta_{1}>0$ and $\beta_{2}>0$ are two parameters. For convenience, let
$$\mathcal L_{\beta_1,\beta_2}^k(u,v,\lambda,\bar x,\bar{y})=\mathcal L_{\beta_1,\beta_2}(u,v,\lambda,\bar x,\bar{y})-\frac{1}{2\beta_1}\|\bar{x}-\bar{x}^k\|^2-\frac{1}{2\beta_2}\|\bar{y}-\bar{y}^k\|^2.$$

%In order to describe the algorithm framework conveniently, we write the above formula as
%\begin{eqnarray*}
%\mathcal L_{\beta_1,\beta_2}(x_1,x_2,x_3,x_4,x_5)=&&f^*(x_1)+x_3^Tb+x_4^T(x_1+A^Tx_2)+x_5^T(x_2-x_3)\\
%&&+\frac{\beta_1}{2}\|x_1+A^Tx_2\|^2+\frac{\beta_2}{2}\|x_2-x_3\|^2.
%\end{eqnarray*}
Now we present the lift-and-permute scheme of ADMM for solving (P).
\begin{sche}\label{sche1}
	~\\
	\noindent\rule[0.25\baselineskip]{\textwidth}{1pt}
	\hspace*{0.02in} {\bf Input:}  \hspace*{0.02in}
	maximum iteration number $K$ and	initial point  $\{u^0,v^0,\lambda^0,\bar{x}^0,\bar{y}^0\}.$\\
	\hspace*{0.02in} {\bf Output:} $\{u^{K+1},v^{K+1},\lambda^{K+1},\bar{x}^{K+1},\bar{y}^{K+1}\}.$\\
	\hspace*{0.02in} Let $\{t_{1}-t_{2}-t_{3}-t_{4}-t_{5}\}$ be a permutation of $\{u,v,\lambda,\bar{x},\bar{y}\}.$\\
	\hspace*{0.02in} For $k=1,2,\cdots,K$   do\\
	1.~~ If $t_{1}\in\{u,v,\lambda\}$, then
	$$t_{1}^{k+1}=\arg\min\limits_{t_1}\mathcal L_{\beta_1,\beta_2}^k(t_1,t_{2}^k,t_{3}^k,t_{4}^k,t_{5}^k),$$
	~~\quad else
	$$t_{1}^{k+1}=\arg\max\limits_{t_1}\mathcal L_{\beta_1,\beta_2}^k(t_1,t_{2}^k,t_{3}^k,t_{4}^k,t_{5}^k).$$
	2.~~ If $t_{2}\in\{u,v,\lambda\}$, then\\
	$$t_{2}^{k+1}=\arg\min\limits_{t_2}\mathcal L_{\beta_1,\beta_2}^k(t_{1}^{k+1},t_2,t_{3}^k,t_{4}^k,t_{5}^k),$$
	~~\quad else
	$$t_{2}^{k+1}=\arg\max\limits_{t_2}\mathcal L_{\beta_1,\beta_2}^k(t_{1}^{k+1},t_2,t_{3}^k,t_{4}^k,t_{5}^k).$$
	3.~~ If  $t_{3}\in\{u,v,\lambda\}$,
	$$t_{3}^{k+1}=\arg\min\limits_{t_3}\mathcal L_{\beta_1,\beta_2}^k(t_{1}^{k+1},t_{2}^{k+1},t_3,t_{4}^k,t_{5}^k),$$
	~~\quad else $$t_{3}^{k+1}=\arg\max\limits_{t_3}\mathcal L_{\beta_1,\beta_2}^k(t_{1}^{k+1},t_{2}^{k+1},t_3,t_{4}^k,t_{5}^k).$$
	4.~~ If $t_{4}\in\{u,v,\lambda\}$,
	$$t_{4}^{k+1}=\arg\min\limits_{t_4}\mathcal L_{\beta_1,\beta_2}^k(t_{1}^{k+1},t_{2}^{k+1},t_{3}^{k+1},t_4,t_{5}^k),$$
	~~\quad else  $$t_{4}^{k+1}=\arg\max\limits_{t_4}\mathcal L_{\beta_1,\beta_2}^k(t_{1}^{k+1},t_{2}^{k+1},t_{3}^{k+1},t_4,t_{5}^k).$$
	5.~~ If $t_{5}\in\{u,v,\lambda\}$,
	$$t_{5}^{k+1}=\arg\min\limits_{t_5}\mathcal L_{\beta_1,\beta_2}^k(t_{1}^{k+1},t_{2}^{k+1},t_{3}^{k+1},t_{4}^{k+1},t_5),$$
	~~\quad else $$t_{5}^{k+1}=\arg\max\limits_{t_5}\mathcal L_{\beta_1,\beta_2}^k(t_{1}^{k+1},t_{2}^{k+1},t_{3}^{k+1},t_{4}^{k+1},t_5).$$
	
	\noindent\rule[0.25\baselineskip]{\textwidth}{1pt}
\end{sche}

Scheme \ref{sche1} contains  $5!=120$ algorithms. We can always assume $t_1=u$, since otherwise, we can start from a proper initial point and then generate an iterative sequence coinciding with that of Algorithm  $t_1=u$.  So the scheme \ref{sche1} remains $4!=24$  algorithms.

Up to different initial points, we have the following equivalence,
\begin{eqnarray}
\{u-\bar x-v-\lambda-\bar y\} &\Leftrightarrow&
\{\bar y-u-\bar x-v-\lambda\}\nonumber\\
&\Leftrightarrow&
\{u-\bar y-\bar x-v-\lambda\}\Leftrightarrow
\{u-\bar x-\bar y-v-\lambda\} \label{uyx:1}\\
\{u-\bar x-v-\lambda-\bar y\} &\Leftrightarrow&
\{\lambda-\bar y-u-\bar x-v\}
\Leftrightarrow
\{u-\lambda-\bar y-\bar x-v\}\label{uyx:2}\\
&\Leftrightarrow&
\{u-\lambda-\bar x-\bar y-v\}\Leftrightarrow\{u-\bar x-\lambda-\bar y-v\}\label{uyx:3}
\end{eqnarray}
where \eqref{uyx:1}, the last equivalence in \eqref{uyx:2}, and \eqref{uyx:3} hold since the update  of either  $\bar y$ or $\lambda$ is independent of  $\bar x$ and $u$. In the similar way as above, we can show that all the $24$  algorithms (see the first column in Table \ref{tab:1}) can be classified into four categories as listed in  the second column in  Table \ref{tab:1}.

The equivalence among the remaind four algorithms and (dual-primal) balanced ALM  are summarized in Columns \uppercase\expandafter{\romannumeral2}-\uppercase\expandafter{\romannumeral4}  in Table \ref{tab:1}. The corresponding proofs are given in the next two subsections, respectively.

%These 24 update orders can be divided into four categories. And these four update orders correspond to two algorithms which contain the state-of-the-art algorithms for solving problem (P). Moreover these two algorithms are equivalent to the same algorithm: Douglas-Rachford splitting algorithm.
%Our findings are summarized in Table \ref{tab:1}. Below we prove the marvellous equivalences contained in Table \ref{tab:1}.

\begin{table}[htb]\centering
	\caption{Algorithms in the same row are equivalent to each other. } \label{tab:1}
	\begin{center}
		\begin{tabular}{|c|c|c|c|c|} \hline
\uppercase\expandafter{\romannumeral1}&\uppercase\expandafter{\romannumeral2} &\uppercase\expandafter{\romannumeral3} &\uppercase\expandafter{\romannumeral4} &\uppercase\expandafter{\romannumeral5} \\ \hline
			$u-\bar x-\lambda-\bar y-v$& \multirow{5}*{$u-\bar x-v-\lambda-\bar y$}& \multirow{10}*{balanced ALM}&
				\multirow{20}*{\makecell{proximal \\ADMM}}&
					\multirow{20}*{DRS}\\ \cline{1-1}
			$u-\bar x-\bar y-v-\lambda$& & & & \\ \cline{1-1}
			$u-\lambda-\bar x-\bar y-v$& & & & \\ \cline{1-1}
			$u-\lambda-\bar y-\bar x-v$& & & & \\ \cline{1-1}
			$u-\bar y-\bar x-v-\lambda$& & & & \\ \cline{1-2}
			$u-\bar x-v -\bar y-\lambda$& \multirow{5}*{$u-\bar x-\lambda-v-\bar y$}& & & \\ \cline{1-1}
			$u-\bar x-\bar y-\lambda-v$& & & & \\ \cline{1-1}
			$u-\lambda-\bar x-v-\bar y$& & & & \\ \cline{1-1}
			$u-\bar y-\bar x-\lambda-v$& & & & \\ \cline{1-1}
			$u-\bar y-\lambda-\bar x-v$& & & & \\ \cline{1-3}
			$u-v-\lambda-\bar y-\bar x$& \multirow{5}*{$u-v-\lambda-\bar x-\bar y$}& \multirow{10}*{\makecell{dual-primal \\ balanced ALM}}& & \\ \cline{1-1}
			$u-v-\bar x-\lambda-\bar y$& & & &\\ \cline{1-1}
			$u-\lambda-\bar y-v-\bar x$& & & &\\ \cline{1-1}
			$u-\bar y-v-\bar x-\lambda$& & & &\\ \cline{1-1}
			$u-\bar y-v-\lambda-\bar x$& & & &\\ \cline{1-2}
			$u-v-\bar x-\bar y-\lambda$& \multirow{5}*{$u-\lambda-v-\bar x-\bar y$}& & &\\ \cline{1-1}
			$u-v-\bar y-\lambda-\bar x$& & & &\\ \cline{1-1}
			$u-v-\bar y-\bar x-\lambda$& & & &\\ \cline{1-1}
			$u-\lambda-v-\bar y-\bar x$& & & &\\ \cline{1-1}
			$u-\bar y-\lambda-v-\bar x$& & & &\\ \hline
			
		\end{tabular}
	\end{center}
\end{table}

\subsection{Equivalence between Columns \uppercase\expandafter{\romannumeral2} and \uppercase\expandafter{\romannumeral3}}
\subsubsection{Equivalence between Algorithm $\{u-\bar{x}-v-\lambda-\bar{y}\}$ and balanced ALM }

We first write down the algorithm in the scheme \ref{sche1} corresponding to the  order $\{u-\bar{x}-v -\lambda -\bar{y}\}$.

\begin{algo}[Algorithm $\{u-\bar{x}-v-\lambda-\bar{y}\}$]
	\label{algo1}
	\begin{subnumcases}
	~u^{k+1}=\arg\min\limits_{u}\mathcal L_{\beta_1,\beta_2}(u,v^k,\lambda^k,\bar{x}^k,\bar{y}^k),\label{algo1:1}\\
	\bar{x}^{k+1}=\bar{x}^k+\beta_1(u^{k+1}+A^Tv^k),\\
	v^{k+1}=\arg\min\limits_{v}\mathcal L_{\beta_1,\beta_2}(u^{k+1},v,\lambda^k,\bar{x}^{k+1},\bar{y}^{k}),\\
	\lambda^{k+1}=\arg\min\limits_{\lambda}\mathcal L_{\beta_1,\beta_2}(u^{k+1},v^{k+1},\lambda,\bar{x}^{k+1},\bar{y}^{k}),\\
	\bar{y}^{k+1}=\bar{y}^k+\beta_2(v^{k+1}-\lambda^{k+1}).\label{algo1:2}
	\end{subnumcases}
\end{algo}

Surprisingly, we can show that  Algorithm \ref{algo1}  is in fact equivalent to balanced ALM \eqref{BALM} for solving (P).

According to optimality conditions, we first simplify \eqref{algo1:1}-\eqref{algo1:2} as
\begin{subnumcases}
~0\in\partial f^*(u^{k+1})+\bar{x}^k+\beta_1(u^{k+1}+A^Tv^k)\label{11},\\
\bar{x}^{k+1}=\bar{x}^k+\beta_1(u^{k+1}+A^Tv^k)\label{12},\\
0=A\bar{x}^{k+1}+\bar{y}^k+\beta_1Au^{k+1}+\beta_1AA^Tv^{k+1}+\beta_2(v^{k+1}-\lambda^k)\label{13},\\
0=b-\bar{y}^k-\beta_2(v^{k+1}-\lambda^{k+1})\label{14},\\
\bar{y}^{k+1}=\bar{y}^k+\beta_2(v^{k+1}-\lambda^{k+1}).\label{15}
\end{subnumcases}

\begin{lem}\label{lem1}
	For Algorithm \ref{algo1}, we have $v^{k+1}\equiv\lambda^{k+1}$ and $\bar{y}^{k+1}\equiv b$ for all $k\ge0$.
\end{lem}
\begin{proof}
For any $k\ge0$, it follows from \eqref{14} and\eqref{15} that
\begin{eqnarray}\label{bar-y=b}
\begin{split}
b&=\bar{y}^kk+\beta_2(v^{k+1}-\lambda^{k+1})=\bar{y}^{k+1}.
\end{split}
\end{eqnarray}
Substituting \eqref{bar-y=b} into \eqref{14} yields
$
v^{k+1}=\lambda^{k+1}.
$
The proof is complete.
\end{proof}

%The following theorem establishes the equivalence between balanced ALM \eqref{BALM} and Algorithm \ref{algo1}.
\begin{thm}\label{thm:m1}
	Balanced ALM \eqref{BALM} for solving  (P) is equivalent to Algorithm \ref{algo1} with the special parametric settings $\beta_1=1/r$ and $\beta_2=\delta$.
\end{thm}
\begin{proof}
Define $x^{k}:=-\bar{x}^{k}$ for all $k\geq0$. By \eqref{12}, we have
\begin{equation}
\bar{x}^k+\beta_1(u^{k+1}+A^Tv^k)=-x^{k+1}\label{111}.
\end{equation}
Substituting \eqref{111} into \eqref{11} yields that
\begin{eqnarray}
&&0\in\partial f^*(u^{k+1})-x^{k+1}\nonumber\\&\Longleftrightarrow&u^{k+1}\in\partial f(x^{k+1})\nonumber\\
&\Longleftrightarrow&0\in\partial f(x^{k+1})+\frac{1}{\beta_1}(x^{k+1}-x^k)+A^Tv^k\nonumber\\
&\Longleftrightarrow&0\in\partial f(x^{k+1})+\frac{1}{\beta_1}(x^{k+1}-x^k)+A^T\lambda^k,\label{112}
\end{eqnarray}
where the first equivalence holds since
$
x\in\partial f^*(u)\Longleftrightarrow u\in\partial f(x),
$
and equation \eqref{112} follows from Lemma \ref{lem1}.
With the setting $\beta_1=1/r$,  \eqref{112} is exactly the optimality condition of the $x$-subproblem of balanced ALM \eqref{BALM}.

Moreover, it follows from \eqref{111} that
\begin{equation}\label{thm:p0}
u^{k+1}=-\frac{1}{\beta_1}x^{k+1}+\frac{1}{\beta_1}x^k-A^Tv^k.
\end{equation}
By taking the setting $\beta_2=\delta$ into \eqref{13}, we have
\begin{eqnarray}
0&=&(\beta_1AA^T+\beta_2I)v^{k+1}+A\bar{x}^{k+1}+\beta_1Au^{k+1}-\beta_2\lambda^k+\bar{y}^k \label{thm:p1}\\
&=&(\beta_1AA^T+\beta_2I)v^{k+1}-2Ax^{k+1}+Ax^k-\beta_1AA^T v^k-\beta_2v^k+\bar{y}^k \label{thm:p2}\\
&=&(\beta_1AA^T+\beta_2I)(v^{k+1}-v^k)-[A(2x^{k+1}-x^k)-b] \nonumber\\
&=&(\frac{1}{r}AA^T+\delta I)(\lambda^{k+1}-\lambda^k)-[A(2x^{k+1}-x^k)-b] \label{thm:p4},
\end{eqnarray}
where equation \eqref{thm:p2} is obtained by substituting \eqref{thm:p0} into \eqref{thm:p1} and the equation \eqref{thm:p4} follows from Lemma \ref{lem1}. The proof is complete since the equation \eqref{thm:p4} corresponds to the $\lambda$-subproblem of balanced ALM \eqref{BALM}.
\end{proof}

\subsubsection{Equivalence between Algorithm $\{u-v-\lambda-\bar{x}-\bar{y}\}$ and dual-primal balanced ALM }

By replacing the update order of Algorithm \ref{algo1} with  $\{u-v-\lambda-\bar{x}-\bar{y}\}$, we obtain the following algorithm, which corresponds to the classical ADMM in solving the three-block convex optimization problem.
\begin{algo}[Algorithm $\{u-v-\lambda-\bar{x}-\bar{y}\}$]
	\label{algo2}
\[	\begin{cases}
	u^{k+1}=\arg\min\limits_{u}\mathcal L_{\beta_1,\beta_2}(u,v^k,\lambda^k,\bar{x}^k,\bar{y}^k),\nonumber\\
	v^{k+1}=\arg\min\limits_{v}\mathcal L_{\beta_1,\beta_2}(u^{k+1},v,\lambda^k,\bar{x}^{k},\bar{y}^{k}),\nonumber\\
	\lambda^{k+1}=\arg\min\limits_{\lambda}\mathcal L_{\beta_1,\beta_2}(u^{k+1},v^{k+1},\lambda,\bar{x}^{k},\bar{y}^{k}),\nonumber\\
	\bar{x}^{k+1}=\bar{x}^k+\beta_1(u^{k+1}+A^Tv^{k+1}),\nonumber\\
	\bar{y}^{k+1}=\bar{y}^k+\beta_2(v^{k+1}-\lambda^{k+1}).  \nonumber
	\end{cases}\]
\end{algo}
%It is interesting to see that  Algorithm \ref{algo2} is applying the classical ADMM to solve the three-block convex optimization problem.

Based on optimality conditions, we can rewrite Algorithm \ref{algo2} as
\begin{subnumcases}~
0\in\partial f^*(u^{k+1})+\bar{x}^k+\beta_1(u^{k+1}+A^Tv^k)\label{A21},\\
0=A\bar{x}^k+\bar{y}^k+\beta_1Au^{k+1}+\beta_1AA^Tv^{k+1}+\beta_2(v^{k+1}-\lambda^k)\label{A22},\\
0=b-\bar{y}^k-\beta_2(v^{k+1}-\lambda^{k+1})\label{A23},\\
\bar{x}^{k+1}=\bar{x}^k+\beta_1(u^{k+1}+A^Tv^{k+1})\label{A24},\\
\bar{y}^{k+1}=\bar{y}^k+\beta_2(v^{k+1}-\lambda^{k+1}).\label{A25}
\end{subnumcases}
%With a proof similar to that of Lemma \ref{lem1}, we can show the following result.
To reveal the equivalence, we first need an observation similar to Lemma \ref{lem1}.
\begin{lem}\label{lem2}
	For Algorithm \ref{algo2}, we have $v^{k+1}=\lambda^{k+1}$ and $\bar{y}^{k+1}=b$ for all $k\ge0$.
\end{lem}
%The equivalence between Algorithm \ref{algo2} and dual-primal balanced ALM is as follows.
\begin{thm}\label{thm2}
With the special parametric settings $\beta_1=1/r$ and $\beta_2=\delta$,
	Algorithm \ref{algo2}  is equivalent to dual-primal balanced ALM for solving (P).
\end{thm}
\begin{proof}
For any $k\ge1$, define
\begin{equation}\label{115}
x^{k+1}:=-\bar{x}^k-\beta_1u^{k+1}-\beta_1A^Tv^k,
\end{equation}
which implies that
\begin{eqnarray}
u^{k+1}&=&-\frac{1}{\beta_1}\bar{x}^k-\frac{1}{\beta_1}x^{k+1}-A^Tv^k.
\label{119}
\end{eqnarray}
Substituting \eqref{119} into \eqref{A24} to replace $u^{k+1}$ yields that
\begin{equation}
\bar{x}^{k+1}% =\bar{x}^k+\beta_1(u^{k+1}+A^Tv^{k+1})\nonumber\\
=-x^{k+1}+\beta_1A^T(v^{k+1}-v^k).\label{118}
\end{equation}
%As $x\in\partial f^*(u)\Longleftrightarrow u\in\partial f(x)$, we substitute \eqref{115} and \eqref{118} into \eqref{21} to get
Then, we have the following reformulations of \eqref{A21}:
\begin{eqnarray}
&&0\in\partial f^*(u^{k+1})+\bar{x}^k+\beta_1(u^{k+1}+A^Tv^k) \nonumber\\
\Longleftrightarrow &&0\in\partial f^*(u^{k+1})-x^{k+1} \label{A2:1}\\
\Longleftrightarrow &&0\in\partial f(x^{k+1})-u^{k+1} \label{A2:2}\\
\Longleftrightarrow &&0\in\partial f(x^{k+1})+\frac{1}{\beta_1}\bar{x}^k+\frac{1}{\beta_1}x^{k+1}+A^Tv^k \label{A2:3}\\
\Longleftrightarrow &&0\in\partial f(x^{k+1})+\frac{1}{\beta_1}(x^{k+1}-x^k)+A^T(2v^k-v^{k-1}) \label{A2:4}\\
\Longleftrightarrow &&0\in\partial f(x^{k+1})+r(x^{k+1}-x^k)+A^T(2\lambda^k-\lambda^{k-1}) \label{222}.
\end{eqnarray}
where the equation \eqref{A2:1} follows from substituting  \eqref{115} into
\eqref{A21}, the equations \eqref{A2:3} and \eqref{A2:4} are obtained by substituting \eqref{119} and \eqref{118} (with $k:=k-1$) into \eqref{A2:2} and \eqref{A2:3}, respectively, the
equation \eqref{222} is due to Lemma \ref{lem2}.

Multiplying both sides of \eqref{A24}  by $A$ yields that
\begin{eqnarray}\label{221}
A(\bar{x}^k+\beta_1u^{k+1})=A\bar{x}^{k+1}-\beta_1AA^Tv^{k+1}.
\end{eqnarray}
Then,  we can simplify \eqref{A22} as follows:
\begin{eqnarray}
&&0=A\bar{x}^k+\bar{y}^k+\beta_1Au^{k+1}+\beta_1AA^Tv^{k+1}+\beta_2(v^{k+1}-\lambda^k) \nonumber\\
\Longleftrightarrow &&0=(\beta_1AA^T+\beta_2I)v^{k+1}+A(\bar{x}^k+\beta_1 u^{k+1})+\bar{y}^k
-\beta_2\lambda^k \nonumber\\
\Longleftrightarrow &&0=(\beta_1AA^T+\beta_2I)v^{k+1}+A\bar{x}^{k+1}-\beta_1AA^Tv^{k+1}+\bar{y}^k-\beta_2\lambda^k \label{A22:1}\\
\Longleftrightarrow &&0=(\beta_1AA^T+\beta_2I)(v^{k+1}-v^k)-Ax^{k+1}+b \label{A22:2}\\
\Longleftrightarrow &&(\frac{1}{r}AA^T+\delta I)(\lambda^{k+1}-\lambda^k)=Ax^{k+1}-b, \label{A22:3}
\end{eqnarray}
where the equation \eqref{A22:1} follows from \eqref{221}, the
equation \eqref{A22:2} is obtained from substituting \eqref{118} into  \eqref{A22:1}, and the equation \eqref{A22:3}
follows from Lemma \ref{lem2}.

We complete the proof by combing both \eqref{222} and \eqref{A22:3} with \eqref{311}.
\end{proof}

\subsubsection{Equivalence between Algorithm $\{u-\bar{x}-\lambda-v-\bar{y}\}$ and balanced ALM}
%Consider the update order $\{u-\bar{x}-\lambda-v-\bar{y}\}$ in Scheme \ref{sche1}.
\begin{algo}[Algorithm $\{u-\bar{x}-\lambda-v-\bar{y}\}$]
	\label{algo4}
\[	\begin{cases}
	u^{k+1}=\arg\min\limits_{u}\mathcal L_{\beta_1,\beta_2}(u,v^k,\lambda^k,\bar{x}^k,\bar{y}^k),\nonumber\\
	\bar{x}^{k+1}=\bar{x}^k+\beta_1(u^{k+1}+A^Tv^k),\nonumber\\
	\lambda^{k+1}=\arg\min\limits_{\lambda}\mathcal L_{\beta_1,\beta_2}(u^{k+1},v^k,\lambda,\bar{x}^{k+1},\bar{y}^{k}),\nonumber\\
	v^{k+1}=\arg\min\limits_{v}\mathcal L_{\beta_1,\beta_2}(u^{k+1},v,\lambda^{k+1},\bar{x}^{k+1},\bar{y}^{k}),\nonumber\\ \bar{y}^{k+1}=\bar{y}^k+\beta_2(v^{k+1}-\lambda^{k+1}).\nonumber
	\end{cases}
\]
\end{algo}

By optimality conditions, we can rewrite Algorithm \ref{algo4} as follows:
\begin{subnumcases}
~0\in\partial f^*(u^{k+1})+\bar{x}^k+\beta_1(u^{k+1}+A^Tv^k),\label{41}\\
\bar{x}^{k+1}=\bar{x}^k+\beta_1(u^{k+1}+A^Tv^k),\label{42}\\
0=b-\bar{y}^k-\beta_2(v^k-\lambda^{k+1}),\label{43}\\
0=A\bar{x}^{k+1}+\bar{y}^k+\beta_1Au^{k+1}+\beta_1AA^Tv^{k+1}+\beta_2(v^{k+1}-\lambda^{k+1}),\label{44}\\
\bar{y}^{k+1}=\bar{y}^k+\beta_2(v^{k+1}-\lambda^{k+1}).\label{45}
\end{subnumcases}

Similarly we can prove that Algorithm \ref{algo4} is equivalent to balanced ALM.

\begin{thm}\label{thm:m5}
	Balanced ALM \eqref{BALM} for solving problem (P) is equivalent to Algorithm \ref{algo4} with the special parametric settings $\beta_1=1/r$ and $\beta_2=\delta$.
\end{thm}
\begin{proof}
Define $x^{k}:=-\bar{x}^{k}$ for all $k\geq0$. By \eqref{42}, we obtain
\begin{equation}
\bar{x}^k+\beta_1(u^{k+1}+A^Tv^k)=-x^{k+1}\label{411}.
\end{equation}
Substituting \eqref{411} into \eqref{41} yields that
\begin{eqnarray}
0\in\partial f^*(u^{k+1})-x^{k+1}&\Longleftrightarrow&u^{k+1}\in\partial f(x^{k+1})\nonumber\\
&\Longleftrightarrow&0\in\partial f(x^{k+1})+\frac{1}{\beta_1}(x^{k+1}-x^k)+A^Tv^k\nonumber\\
&\Longleftrightarrow&0\in\partial f(x^{k+1})+r(x^{k+1}-x^k)+A^Tv^k.\label{412}
\end{eqnarray}
%With the setting $\beta_1=1/r$,  \eqref{112} is exactly the optimality conditions of the $x$-subproblem of balanced ALM \eqref{BALM}.
Moreover, we have
\begin{eqnarray}
&&u^{k+1}=-\frac{1}{\beta_1}x^{k+1}+\frac{1}{\beta_1}x^k-A^Tv^k,\label{thm:p40}\\
&&\bar y^k=\bar y^{k+1}-\beta_2(v^{k+1}-\lambda^{k+1}),\label{thm:p41}
\end{eqnarray}
which follows from \eqref{411} and \eqref{45}, respectively.
Substituting \eqref{thm:p41} into \eqref{43} yields that
\[%\begin{equation}\label{thm:p42}
\bar y^{k+1}=b+\beta_2(v^{k+1}-v^k).
\]%\end{equation}
By substituting \eqref{thm:p40} and \eqref{thm:p41} into \eqref{44}, respectively, we can obtain
\begin{eqnarray}
0&=&A\bar{x}^{k+1}+\left[\bar{y}^k+\beta_2(v^{k+1}-\lambda^{k+1})\right]+\beta_1Au^{k+1}+\beta_1AA^Tv^{k+1}\nonumber\\
&=&-Ax^{k+1}+\bar y^{k+1}+\beta_1Au^{k+1}+\beta_1AA^Tv^{k+1}\nonumber\\
&=&-Ax^{k+1}+b+\beta_2(v^{k+1}-v^k)-Ax^{k+1}+Ax^k-\beta_1AA^Tv^k+\beta_1AA^Tv^{k+1}\nonumber\\
&=&(\beta_1AA^T+\beta_2I)(v^{k+1}-v^k)-[A(2x^{k+1}-x^k)-b]\nonumber\\
&=&(\frac{1}{r}AA^T+\delta I)(v^{k+1}-v^k)-[A(2x^{k+1}-x^k)-b].\label{thm:p43}
\end{eqnarray}
Notice that \eqref{412} and \eqref{thm:p43} are  optimality conditions of balanced ALM.
\end{proof}

\subsubsection{Equivalence between Algorithm $\{u-\lambda-v-\bar{x}-\bar{y}\}$ and dual-primal balanced ALM}

%Considering the update order $\{u-\lambda-v-\bar{x}-\bar{y}\}$  in Scheme \ref{sche1}.
\begin{algo}[Algorithm $\{u-\lambda-v-\bar{x}-\bar{y}\}$]
	\label{algo5}
\[
\begin{cases}
	u^{k+1}=\arg\min\limits_{u}\mathcal L_{\beta_1,\beta_2}(u,v^k,\lambda^k,\bar{x}^k,\bar{y}^k),\nonumber\\
	\lambda^{k+1}=\arg\min\limits_{\lambda}\mathcal L_{\beta_1,\beta_2}(u^{k+1},v^k,\lambda,\bar{x}^{k},\bar{y}^{k}),\nonumber\\
	v^{k+1}=\arg\min\limits_{v}\mathcal L_{\beta_1,\beta_2}(u^{k+1},v,\lambda^{k+1},\bar{x}^{k},\bar{y}^{k}),\nonumber\\ \bar{x}^{k+1}=\bar{x}^k+\beta_1(u^{k+1}+A^Tv^{k+1}),\nonumber\\ \bar{y}^{k+1}=\bar{y}^k+\beta_2(v^{k+1}-\lambda^{k+1}).\nonumber
	\end{cases}
\]
\end{algo}

Based on optimality conditions, we can rewrite Algorithm \ref{algo5} as
\begin{subnumcases}
~0\in\partial f^*(u^{k+1})+\bar{x}^k+\beta_1(u^{k+1}+A^Tv^k),\label{A61}\\
0=b-\bar{y}^k-\beta_2(v^k-\lambda^{k+1}),\label{A62}\\
0=A\bar{x}^k+\bar{y}^k+\beta_1Au^{k+1}+\beta_1AA^Tv^{k+1}+\beta_2(v^{k+1}-\lambda^{k+1}),\label{A63}\\
\bar{x}^{k+1}=\bar{x}^k+\beta_1(u^{k+1}+A^Tv^{k+1}),\label{A64}\\
\bar{y}^{k+1}=\bar{y}^k+\beta_2(v^{k+1}-\lambda^{k+1}).\label{A65}
\end{subnumcases}
We reveal the equivalence between Algorithm \ref{algo5} and dual-primal balanced ALM.

\begin{thm}\label{thm6}
	Algorithm \ref{algo5}  is equivalent to  dual-primal balanced ALM for solving (P) with the special parametric settings $\beta_1=1/r$ and $\beta_2=\delta$.
\end{thm}
\begin{proof}
For any $k\ge1$, we define
\begin{equation}\label{615}
x^{k+1}:=-\bar{x}^k-\beta_1u^{k+1}-\beta_1A^Tv^k,
\end{equation}
which implies that
\begin{eqnarray}
%\bar{x}^k&=&-x^{k+1}-\beta_1u^{k+1}-\beta_1A^Tv^k,\label{617}\\
u^{k+1}&=&-\frac{1}{\beta_1}\bar{x}^k-\frac{1}{\beta_1}x^{k+1}-A^Tv^k.\label{619}
\end{eqnarray}
Substituting \eqref{619} into \eqref{A64} to replace $u^{k+1}$ yields that
\begin{equation}
\bar{x}^{k+1}% =\bar{x}^k+\beta_1(u^{k+1}+A^Tv^{k+1})\nonumber\\
=-x^{k+1}+\beta_1A^T(v^{k+1}-v^k).\label{618}
\end{equation}
%As $x\in\partial f^*(u)\Longleftrightarrow u\in\partial f(x)$, we substitute \eqref{115} and \eqref{118} into \eqref{21} to get
Then, we have the following reformulations of \eqref{A61}:
\begin{eqnarray}
&&0\in\partial f^*(u^{k+1})+\bar{x}^k+\beta_1(u^{k+1}+A^Tv^k) \nonumber\\
\Longleftrightarrow &&0\in\partial f^*(u^{k+1})-x^{k+1} \label{A6:1}\\
\Longleftrightarrow &&0\in\partial f(x^{k+1})-u^{k+1} \label{A6:2}\\
\Longleftrightarrow &&0\in\partial f(x^{k+1})+\frac{1}{\beta_1}\bar{x}^k+\frac{1}{\beta_1}x^{k+1}+A^Tv^k \label{A6:3}\\
\Longleftrightarrow &&0\in\partial f(x^{k+1})+r(x^{k+1}-x^k)+A^T(2v^k-v^{k-1}), \label{A6:4}
\end{eqnarray}
where the equation \eqref{A6:1} follows from substituting  \eqref{615} into
\eqref{A61}, the equations \eqref{A6:3} and \eqref{A6:4} are obtained by substituting \eqref{619} and \eqref{618} (with $k:=k-1$) into \eqref{A6:2} and \eqref{A6:3}, respectively.

Multiplying both sides of \eqref{A64}  by $A$ yields that
\begin{eqnarray}\label{621}
A(\bar{x}^k+\beta_1u^{k+1})=A\bar{x}^{k+1}-\beta_1AA^Tv^{k+1}.
\end{eqnarray}
%Shifting the terms of \eqref{A65} we can get
%\begin{eqnarray}
%\bar y^k=\bar y^{k+1}-\beta_2(v^{k+1}-\lambda^{k+1}).\label{623}
%\end{eqnarray}
By substituting \eqref{A65} into \eqref{A62}, we have
\begin{eqnarray}
\bar y^{k+1}=b+\beta_2(v^{k+1}-v^k).\label{625}
\end{eqnarray}
Then,  we can simplify \eqref{A63} as follows:
\begin{eqnarray}
&&0=A\bar{x}^k+\bar{y}^k+\beta_1Au^{k+1}+\beta_1AA^Tv^{k+1}+\beta_2(v^{k+1}-\lambda^{k+1}) \nonumber\\
\Longleftrightarrow &&0=A(\bar{x}^k+\beta_1 u^{k+1})+\left[\bar y^k +\beta_2(v^{k+1}-\lambda^{k+1})\right]+\beta_1AA^Tv^{k+1}
\nonumber\\
\Longleftrightarrow &&0=A\bar{x}^{k+1}+\bar y^{k+1}
\label{A62:1}\\
\Longleftrightarrow &&0=-Ax^{k+1}+\beta_1AA^T(v^{k+1}-v^k)+b+\beta_2(v^{k+1}-v^k)
\label{A62:2}\\
\Longleftrightarrow &&(\beta_1AA^T+\beta_2I)(v^{k+1}-v^k)=Ax^{k+1}-b \nonumber\\
\Longleftrightarrow &&(\frac{1}{r}AA^T+\delta I)(v^{k+1}-v^k)=Ax^{k+1}-b\label{A62:4},
\end{eqnarray}
where the equation \eqref{A62:1} follows from \eqref{621} and \eqref{A65},
the equation \eqref{A62:2} is obtained from substituting \eqref{618} and \eqref{625} into  \eqref{A62:1}.

We complete the proof by combing both \eqref{A6:4} and \eqref{A62:4} with \eqref{311}.
\end{proof}

\subsection{Equivalence between Columns \uppercase\expandafter{\romannumeral3} and \uppercase\expandafter{\romannumeral4}}
Different from the reformulation \eqref{dual1}, we rewrite
the dual problem \eqref{dual} as the following compact version:
\begin{equation}\label{dual22}
\min\limits_{u,\lambda} \left\{ f^*(u)+\lambda^Tb:~A^T\lambda+u=0 \right\}.
\end{equation}
Using Lemma \ref{lem1}, we can reduce the optimality conditions \eqref{11}-\eqref{15} %of Scheme \ref{sche1} in the order of $\{u,\bar{x},v,\lambda,\bar{y}\}$
to
\[
\begin{cases}
0\in\partial f^*(u^{k+1})+\bar{x}^k+\beta_1(u^{k+1}+A^T\lambda^k),\nonumber\\
\bar{x}^{k+1}=\bar{x}^k+\beta_1(u^{k+1}+A^T\lambda^k),\nonumber\\
0=A\bar{x}^{k+1}+b+\beta_1Au^{k+1}+\beta_1AA^T\lambda^{k+1}+\beta_2(\lambda^{k+1}-\lambda^k),\nonumber
\end{cases}
\]
which exactly corresponds to the optimality conditions of the following proximal ADMM for solving the dual problem \eqref{dual22}.
\begin{algo}\label{algo12}
	\begin{eqnarray*}
		\begin{cases}
			u^{k+1}=\arg\min\limits_{u}\{f^*(u)+(\bar{x}^k)^Tu+\frac{\beta_1}{2}\|u+A^T\lambda^k\|^2\},\\
			\bar{x}^{k+1}=\bar{x}^k+\beta_1(u^{k+1}+A^T\lambda^k),\\
			\lambda^{k+1}=\arg\min\limits_{\lambda}\{b^T\lambda+(A\bar{x}^{k+1})^T\lambda+\frac{\beta_1}{2}\|u^{k+1}+A^T\lambda\|^2+\frac{\beta_2}{2}\|\lambda-\lambda^k\|^2\}.
		\end{cases}
	\end{eqnarray*}
\end{algo}
Then, according to Theorem \ref{thm:m1}, we have the following equivalence result.
\begin{cor}
	Balanced ALM \eqref{BALM} for solving (P) is equivalent to the proximal ADMM (Algorithm \ref{algo12}) for solving the dual problem \eqref{dual22} with the special parametric settings $\beta_1=1/r$ and $\beta_2=\delta$.
\end{cor}

Similarly, we can use Lemma \ref{lem2} to further simplify the optimality conditions \eqref{A21}-\eqref{A25} as
\[
\begin{cases}
0\in\partial f^*(u^{k+1})+\bar{x}^k+\beta_1(u^{k+1}+A^T\lambda^k)\label{21},\nonumber\\
0=A\bar{x}^k+b+\beta_1Au^{k+1}+\beta_1AA^T\lambda^{k+1}+\beta_2(\lambda^{k+1}-\lambda^k),\nonumber\\
\bar{x}^{k+1}=\bar{x}^k+\beta_1(u^{k+1}+A^T\lambda^{k+1}),\nonumber
\end{cases}
\]
which exactly corresponds to the optimality conditions of the following proximal ADMM for solving the dual problem \eqref{dual22}.
\begin{algo}
	\begin{eqnarray*}
		\begin{cases}
			u^{k+1}=\arg\min\limits_{u}\{f^*(u)+(\bar{x}^k)^Tu+\frac{\beta_1}{2}\|u+A^T\lambda^k\|^2\},\\
			\lambda^{k+1}=\arg\min\limits_{\lambda}\{b^T\lambda+(A\bar{x}^k)^T\lambda+\frac{\beta_1}{2}\|u^{k+1}+A^T\lambda\|^2+\frac{\beta_2}{2}\|\lambda-\lambda^k\|^2\},\\
			\bar{x}^{k+1}=\bar{x}^k+\beta_1(u^{k+1}+A^T\lambda^{k+1}).
		\end{cases}
	\end{eqnarray*}
\end{algo}
According to Theorem \ref{thm2}, we have the following conclusion.
\begin{cor}
Dual-primal balanced ALM \eqref{311} for solving (P) is equivalent to  the proximal ADMM for solving the dual problem \eqref{dual22} with the special parametric settings $\beta_1=1/r$ and $\beta_2=\delta$.
\end{cor}

\section{Acceleration}
In this section, we first present accelerated balanced ALM and accelerated dual-primal balanced ALM. Then we provide the convergence rate analysis.

Throughout this section, we assume that $f(x)$ is $\mu$-strongly convex $(\mu\geq0)$.
\begin{defi}
$f(x)$ is $\mu$-strongly convex with $\mu\geq0$ if for  all $x,~y\in \mathbb{R}^n$,
$$f(x)\geq f(y)+g^T(x-y)+\frac{\mu}{2}\|x-y\|^2,~~g\in\partial f(y).$$
\end{defi}
\subsection{Accelerated balanced ALM}
We establish convergence rate analysis on the accelerated balanced ALM.
\begin{algo}[Accelerated balanced ALM]\label{Aa:1}
\[
\begin{cases}			
x^{k+1}=\arg\min\limits_{x}
\{f(x)+\frac{r^k}{2}\|x-(x^k-\frac{1}{r^k}A^T\lambda^k)\|^2\},\\
\lambda^{k+1}=\lambda^{k}+(\frac{1}{r^{k+1}}AA^T+\delta^k I_m)^{-1}[A\widetilde{x}^{k+1}-b],\\
\widetilde{x}^{k+1}=x^{k+1}+\theta^{k}(x^{k+1}-x^{k}).
\end{cases}
\]
\end{algo}
\begin{lem}\label{al1}
%Suppose $f(x)$ is $\mu$-strongly convex with $\mu\geq0$.
Let $\delta^k= \delta'/r^{k+1}$ with $\delta'>0$ and $H=AA^T+\delta' I_m$. For the sequence $\{(x^k,\lambda^k)\}$ generated by Algorithm \ref{Aa:1} and any $x\in\mathbb{R}^n$, $\lambda\in\mathbb{R}^m$, it holds that
	\begin{equation}\label{4.1}
	\begin{aligned}
	&f(x^{k+1})+\lambda^T(Ax^{k+1}-b)-f(x)-({\lambda^{k}})^T(Ax-b)\\\leq&
	\frac{r^k}{2}\|x^k-x\|^2-\frac{r^k+\mu}{2}\|x^{k+1}-x\|^2-\frac{r^k}{2}\|x^k-x^{k+1}\|^2+\frac{1}{2r^k}\Big( \|\lambda^{k-1}-\lambda\|_H^2\\-&\|\lambda^{k}-\lambda\|^2_H-\|\lambda^{k-1}-\lambda^{k}\|^2_H
	\Big)
	+(\lambda^k-\lambda)^TA(\widetilde{x}^k-x^{k+1}).
	\end{aligned}\end{equation}
\end{lem}
\begin{proof}
	According to the optimality condition of the $x$-subproblem  in Algorithm  \ref{Aa:1}, we have
\begin{equation}\label{parf}
	0\in\partial f(x^{k+1})+A^T\lambda^k+r^k(x^{k+1}-x^k),
\end{equation}
%or equivalently,
%-(A^T\lambda^k+r^k(x^{k+1}-x^k))\in\partial f(x^{k+1}).
Then by using $\mu$-strongly convexity of $f(x)$ and \eqref{parf}, we can obtain
\begin{equation}\label{festim}
\begin{aligned}
	&f(x^{k+1})-f(x)\leq (A^T\lambda^k+r^k(x^{k+1}-x^k))^T(x-x^{k+1})-\frac{\mu}{2}\|x^{k+1}-x\|^2\\=&\frac{r^k}{2}\|x^k-x\|^2-\frac{r^k+\mu}{2}\|x^{k+1}-x\|^2-\frac{r^k}{2}\|x^k-x^{k+1}\|^2+(\lambda^k)^TA(x-x^{k+1}).
	\end{aligned}
\end{equation}
	According to the update of $\lambda$ in Algorithm  \ref{Aa:1}, we have
\begin{equation}\label{lbdestim}
\begin{aligned} &(\lambda^{k}-\lambda)^Tb=(\lambda^{k}-\lambda)^T\Big(A\widetilde{x}^k-\frac{1}{r^{k}}(AA^T+\delta' I_m)(\lambda^{k}-\lambda^{k-1})\Big)\\=&
	\frac{1}{2r^k}\Big( \|\lambda^{k-1}-\lambda\|_H^2-\|\lambda^{k}-\lambda\|^2_H-\|\lambda^{k-1}-\lambda^{k}\|^2_H
	\Big)+(\lambda^{k}-\lambda)^TA\widetilde{x}^k.
	\end{aligned}
\end{equation}
Putting \eqref{festim} and \eqref{lbdestim} together completes the proof.
\end{proof}

We first study the converge rate of Algorithm  \ref{Aa:1}
with the setting $r^k=r$ and  $\delta= \delta'/r$, which reduces to balanced ALM \eqref{BALM}.

\begin{thm}
Suppose $f(x)$ is convex (but not necessarily strongly convex). For any $x\in\mathbb{R}^n$ and $\lambda\in\mathbb{R}^m$, the sequence $\{(x^k,\lambda^k)\}$ generated by balanced ALM \eqref{BALM} satisfies that
	$$ \begin{aligned}
	&f(\hat{x}^{K})+\lambda^T(A\hat{x}^{K}-b)-f(x)-({\hat{\lambda}}^{K})^T(Ax-b)\\\leq&
\frac{1}{K+1}\Big(\frac{r}{2}\big(\|x^0-x\|^2+\|x^{-1}-x^0\|^2\big)+\frac{1}{2r} \|\lambda^{-1}-\lambda\|_H^2\\&+(\lambda^{-1}-\lambda)^TA(x^{0}-x^{-1})\Big),
	\end{aligned}
	$$
where $\hat{x}^K= (\sum_{k=0}^{K}x^{k+1})/(K+1)$ and $\hat{\lambda}^K=(\sum_{k=0}^{K}\lambda^{k})/(K+1)$.	
\end{thm}
\begin{proof}
According to Lemma \ref{al1}, we have
\begin{eqnarray} &&f(x^{k+1})+\lambda^T(Ax^{k+1}-b)-f(x)-({\lambda^{k}})^T(Ax-b)\nonumber\\
&\leq& \frac{r}{2}\Big(\|x^k-x\|^2-\|x^{k+1}-x\|^2-\|x^k-x^{k+1}\|^2\Big)+\frac{1}{2r}\Big( \|\lambda^{k-1}-\lambda\|_H^2\nonumber\\
&&-\|\lambda^{k}-\lambda\|^2_H-\|\lambda^{k-1}-\lambda^{k}\|^2_H\Big)+(\lambda^k-\lambda)^TA(2x^{k}-x^{k-1}-x^{k+1})\nonumber\\
&\leq&
\frac{r}{2}\Big(\|x^k-x\|^2-\|x^{k+1}-x\|^2+\|x^{k-1}-x^k\|^2-\|x^{k}-x^{k+1}\|^2\Big)\nonumber\\
&&+\frac{1}{2r}\Big( \|\lambda^{k-1}-\lambda\|_H^2-\|\lambda^{k}-\lambda\|^2_H
\Big)\nonumber\\
&&+(\lambda^{k-1}-\lambda)^TA(x^{k}-x^{k-1})-(\lambda^{k}-\lambda)^TA(x^{k+1}-x^{k}),
\label{festm:1}	\end{eqnarray}
where the last inequality holds as
$$\begin{aligned}
	&(\lambda^k-\lambda)^TA(2x^{k}-x^{k-1}-x^{k+1})\\
=&(\lambda^k-\lambda)^TA(x^{k}-x^{k+1})-(\lambda^{k-1}-\lambda)^TA(x^{k-1}-x^{k})\\
&+(\lambda^{k-1}-\lambda^k)^TA(x^{k-1}-x^k)\\
\leq&(\lambda^k-\lambda)^TA(x^{k}-x^{k+1})-(\lambda^{k-1}-\lambda)^TA(x^{k-1}-x^{k})\\
&+\frac{1}{2r}\|\lambda^{k-1}-\lambda^k\|^2_{AA^T}+\frac{r}{2}\|x^{k-1}-x^k\|^2.
	\end{aligned}$$
%$$\begin{aligned}
%&f(x^{k+1})+\lambda^T(Ax^{k+1}-b)-f(x)-({\lambda^{k}})^T(Ax-b)\\
%\end{aligned}$$
By adding all the inequalities \eqref{festm:1} from $k=0$ to $k=K$ and then dividing both sides by $(K+1)$, we obtain
$$ \begin{aligned}
&f(\hat{x}^{K})+\lambda^T(A\hat{x}^{K}-b)-f(x)-({\hat{\lambda}}^{K})^T(Ax-b)\\\leq&
\frac{1}{K+1}\Big(\frac{r}{2}\big(\|x^0-x\|^2-\|x^{K+1}-x\|^2+\|x^{-1}-x^0\|^2-\|x^K-x^{K+1}\|^2\big)\\&+\frac{1}{2r}\big( \|\lambda^{-1}-\lambda\|_H^2-\|\lambda^{K}-\lambda\|^2_H
\big)\\&+(\lambda^{-1}-\lambda)^TA(x^{0}-x^{-1})-(\lambda^{K}-\lambda)^TA(x^{K+1}-x^{K})\Big)\\\leq&\frac{1}{K+1}\Big(\frac{r}{2}\big(\|x^0-x\|^2+\|x^{-1}-x^0\|^2\big)+\frac{1}{2r} \|\lambda^{-1}-\lambda\|_H^2\\&+(\lambda^{-1}-\lambda)^TA(x^{0}-x^{-1})\Big).
\end{aligned}
$$
The proof is complete.
\end{proof}

He and Yuan \cite{he2021balanced} estabilished the same $O(1/K)$ convergence rate based on the convex combination of iteration $\{(x^{k+1},\lambda^{k+1})\}$.
As a contrast,
our analysis is based on the new
convex combination of iteration $\{(x^{k+1},\lambda^{k})\}$,
%comparatively, we give the $O(1/K)$ convergence rate based on the convex combination of iteration $\{(x^{k+1},\lambda^{k})\}$.
%With the help of our new combination, we can
which further helps to establish $O(1/K^2)$ convergence rate for the accelerated balanced ALM (Algorithm \ref{Aa:1}).

\begin{thm}\label{thma1}
Suppose $f(x)$ is $\mu$-strongly convex. For any $x\in\mathbb{R}^n$ and $\lambda\in\mathbb{R}^m$, the sequence $\{(x^k,\lambda^k)\}$ generated by Algorithm \ref{Aa:1} with the setting $\theta^k= r^k/r^{k+1}$ and $(r^k+\mu)r^k\geq (r^{k+1})^2$ satisfies that
$$\begin{aligned}
&\Big(\sum_{k=0}^{K}r^k\Big)\Big(f(\hat{x}^{K})+\lambda^T(A\hat{x}^{K}-b)-f(x)-({{\hat{\lambda}}^{K}})^T(Ax-b)\Big)\\\leq&
\frac{(r^0)^2}{2}\|x^0-x\|^2+\frac{(r^{-1})^2}{2}\|x^{-1}-x^0\|^2+\frac{1}{2} \|\lambda^{-1}-\lambda\|_H^2\\&+r^{-1}(\lambda^{-1}-\lambda)^TA(x^{0}-x^{-1}),
\end{aligned}
$$
where $\hat{x}^K=(\sum_{k=0}^{K}r^kx^{k+1})/(\sum_{k=0}^{K}r^k)$ and $\hat{\lambda}^K=(\sum_{k=0}^{K}r^k\lambda^{k})/
(\sum_{k=0}^{K}r^k)$. In particular,
with the setting $r^k= \mu(k+1)/3$, we have
$$f(\hat{x}^{K})+\lambda^T(A\hat{x}^{K}-b)-f(x)-({{\hat{\lambda}}^{K}})^T(Ax-b)\le O(1/K^2).$$
\end{thm}
\begin{proof}
First we can verify that
\begin{eqnarray}
&&(\lambda^k-\lambda)^TA(\widetilde{x}^k-x^{k+1})\nonumber\\
&=&	(\lambda^k-\lambda)^TA(x^{k}-x^{k+1})-\theta^{k-1}(\lambda^{k-1}-\lambda)^TA(x^{k-1}-x^k)\nonumber\\
&&+\theta^{k-1}(\lambda^{k-1}-\lambda^{k})^TA(x^{k-1}-x^{k})\nonumber\\
&\leq&(\lambda^k-\lambda)^TA(x^{k}-x^{k+1})-\theta^{k-1}(\lambda^{k-1}-\lambda)^TA(x^{k-1}-x^k)\nonumber\\
&&+\frac{1}{2r^k}\|\lambda^{k-1}-\lambda^{k}\|^2_H+\frac{(r^{k-1})^2}{2r^k}\|x^{k-1}-x^{k}\|^2.\label{acc:1}
	\end{eqnarray}
Multiplying both side of \eqref{4.1} by $r^k$ and using the fact $(r^k+\mu)r^k\geq (r^{k+1})^2$ and \eqref{acc:1} yields that
\begin{eqnarray}
&&r^k\Big(f(x^{k+1})+\lambda^T(Ax^{k+1}-b)-f(x)-({\lambda^{k}})^T(Ax-b)\Big)\nonumber\\
&\leq&
\frac{(r^k)^2}{2}\|x^k-x\|^2-\frac{(r^{k+1})^2}{2}\|x^{k+1}-x\|^2+\frac{1}{2}\Big( \|\lambda^{k-1}-\lambda\|_H^2-\|\lambda^{k}-\lambda\|^2_H
\Big)\nonumber\\
&&+\frac{(r^{k-1})^2}{2}\|x^{k-1}-x^{k}\|^2-\frac{(r^k)^2}{2}\|x^k-x^{k+1}\|^2\nonumber\\
&&-r^{k-1}(\lambda^{k-1}-\lambda)^TA(x^{k-1}-x^k)+r^k(\lambda^k-\lambda)^TA(x^{k}-x^{k+1})
. \label{acc:2}
\end{eqnarray}
By adding all the inequalities \eqref{acc:2} from $k=0$ to $k=K$, we obtain
$$ \begin{aligned}
&\Big(\sum_{k=0}^{K}r^k\Big)\Big(f(\hat{x}^{K})+\lambda^T(A\hat{x}^{K}-b)-f(x)-({\hat{\lambda}}^{K})^T(Ax-b)\Big)\\\leq&
\frac{(r^0)^2}{2}\|x^0-x\|^2-\frac{(r^{K+1})^2}{2}\|x^{K+1}-x\|^2+\frac{(r^{-1})^2}{2}\|x^{-1}-x^0\|^2\\&-\frac{(r^{K})^2}{2}\|x^K-x^{K+1}\|^2+\frac{1}{2}\big( \|\lambda^{-1}-\lambda\|_H^2-\|\lambda^{K}-\lambda\|^2_H
\big)\\&+r^{-1}(\lambda^{-1}-\lambda)^TA(x^{0}-x^{-1})-r^K(\lambda^{K}-\lambda)^TA(x^{K+1}-x^{K})\Big)\\\leq&\frac{(r^0)^2}{2}\|x^0-x\|^2+\frac{(r^{-1})^2}{2}\|x^{-1}-x^0\|^2+\frac{1}{2} \|\lambda^{-1}-\lambda\|_H^2\\&+r^{-1}(\lambda^{-1}-\lambda)^TA(x^{0}-x^{-1}).
\end{aligned}
$$
The proof is complete.
\end{proof}

\begin{rem}\label{re2}
%[Equivalence between  Algorithm \ref{Aa:1} and accelerated proximal ADMM]\label{re2}
As shown in Section 3, Algorithm $\{u-\bar{x}-v-\lambda-\bar{y}\}$ is equivalent to balanced ALM and proximal ADMM. It is natural to ask whether our accelerated balanced ALM (Algorithm \ref{Aa:1}) is equivalent to the accelerated proximal ADMM \cite{2016An}.

Let $u^{k+1}:=-A^T\lambda^k-r^k(x^{k+1}-x^k)$ and $\bar{x}^k:=-x^k$. It holds that
\[
\bar{x}^{k+1}=\bar{x}^k+\frac{1}{r^k}(u^{k+1}+A^T\lambda^k).
\]
According to the optimality condition of $x$-subproblem in Algorithm \ref{Aa:1}, we have $u^{k+1}\in \partial f(x^{k+1})\Longleftrightarrow x^{k+1}\in \partial f^*(u^{k+1}).$
	Then it holds that
	$$
	\begin{aligned}
	0\in& \partial f^*(u^{k+1})-x^{k+1}=\partial f^*(u^{k+1})+\bar{x}^k+\frac{1}{r^k}(u^{k+1}+A^T\lambda^k).
	\end{aligned}$$
	According to the $\lambda$-subproblem in Algorithm \ref{Aa:1}, we obtain
	$$\begin{aligned}
	0=&-[A(x^{k+1}+\theta^k(x^{k+1}-x^k))-b]+(\frac{1}{r^{k+1}}AA^T+\delta^k I_m)(\lambda^{k+1}-\lambda^k)\\=&
	A\bar{x}^{k+1}+b+\frac{1}{r^{k+1}}(Au^{k+1}+AA^T\lambda^{k+1})+\delta^k(\lambda^{k+1}-\lambda^k).
	\end{aligned}
	$$
Therefore, Algorithm \ref{Aa:1} is equivalent to the following proximal ADMM for \eqref{dual22}.
		\begin{eqnarray*}
			\begin{cases}
				u^{k+1}=\arg\min\limits_{u}\{f^*(u)+(\bar{x}^k)^Tu+\frac{1}{2r^k}\|u+A^T\lambda^k\|^2\},\\
				\bar{x}^{k+1}=\bar{x}^k+\frac{1}{r^k}(u^{k+1}+A^T\lambda^k),\\
				\lambda^{k+1}=\arg\min\limits_{\lambda}\{b^T\lambda+(A\bar{x}^{k+1})^T\lambda+\frac{1}{2r^{k+1}}\|u^{k+1}+A^T\lambda\|^2+\frac{\delta^k}{2}\|\lambda-\lambda^k\|^2\}.
			\end{cases}
		\end{eqnarray*}
With the above observation,  one can alternatively establish the $O(1/K^2)$ convergence rate of Algorithm \ref{Aa:1} by referring to the analysis on the general accelerated proximal ADMM.
\end{rem}
\subsection{Accelerated dual-primal balanced ALM}
\begin{algo}[Accelerated dual-primal balanced ALM]\label{Aa:2}
\[
\begin{cases}			
x^{k+1}=\arg\min\limits_{x}
f(x)+\frac{r^k}{2}\|x-(x^k-\frac{1}{r^k}A^T\widetilde{\lambda}^k\|^2\},\\
\lambda^{k+1}=\lambda^{k}+(\frac{1}{r^{k}}AA^T+\delta^k I_m)^{-1}[Ax^{k+1}-b],\\
\widetilde{\lambda}^k=\lambda^k+\theta^{k-1}(\lambda^k-\lambda^{k-1}).
\end{cases}
\]
\end{algo}
\begin{lem}\label{al2}
Let $\delta^k= \delta'/r^{k}$ with $\delta'>0$ and $H=AA^T+\delta' I_m$. For the sequence $\{(x^k,\lambda^k)\}$ generated by Algorithm \ref{Aa:2} and any $x\in\mathbb{R}^n$, $\lambda\in\mathbb{R}^m$, it holds that
	\begin{equation}\label{acc2}
	\begin{aligned}
	&f(x^{k+1})+\lambda^T(Ax^{k+1}-b)-f(x)-({\lambda^{k+1}})^T(Ax-b)\\\leq&
	\frac{r^k}{2}\|x^k-x\|^2-\frac{r^k+\mu}{2}\|x^{k+1}-x\|^2-\frac{r^k}{2}\|x^k-x^{k+1}\|^2+\frac{1}{2r^k}\Big( \|\lambda^{k}-\lambda\|_H^2\\&-\|\lambda^{k+1}-\lambda\|^2_H-\|\lambda^{k}-\lambda^{k+1}\|^2_H
	\Big)
	+(\widetilde{\lambda}^k-\lambda^{k+1})^TA(x^{k+1}-x).
	\end{aligned}\end{equation}
\end{lem}
\begin{proof}
	According to the optimality condition of $x$-subproblem in Algorithm \ref{Aa:2}, we have
\begin{equation}\label{parf2}
	0\in\partial f(x^{k+1})+A^T\widetilde{\lambda}^k+r^k(x^{k+1}-x^k).
\end{equation}
	Then based on $\mu$-strongly convexity of $f(x)$ and \eqref{parf2}, we can obtain
\begin{equation}\label{festm:2}
\begin{aligned}
	&f(x^{k+1})-f(x)\leq (A^T\widetilde{\lambda}^k+r^k(x^{k+1}-x^k))^T(x-x^{k+1})-\frac{\mu}{2}\|x^{k+1}-x\|^2\\=&\frac{r^k}{2}\|x^k-x\|^2-\frac{r^k+\mu}{2}\|x^{k+1}-x\|^2-\frac{r^k}{2}\|x^k-x^{k+1}\|^2+(\widetilde{\lambda}^k)^TA(x-x^{k+1}).
	\end{aligned}
\end{equation}
	According to the  $\lambda$-subproblem  in Algorithm  \ref{Aa:2}, we have
\begin{equation}\label{lbdestm:2}
\begin{aligned} &(\lambda^{k+1}-\lambda)^Tb=(\lambda^{k+1}-\lambda)^T(Ax^{k+1}-\frac{1}{r^{k}}(AA^T+\delta' I_m)(\lambda^{k+1}-\lambda^{k}))\\=&
	\frac{1}{2r^k}\Big( \|\lambda^{k}-\lambda\|_H^2-\|\lambda^{k+1}-\lambda\|^2_H-\|\lambda^{k}-\lambda^{k+1}\|^2_H
	\Big)+(\lambda^{k+1}-\lambda)^TAx^{k+1}.
	\end{aligned}
\end{equation}
Putting \eqref{festm:2} and \eqref{lbdestm:2} together completes the proof.
\end{proof}
\begin{thm}
	Suppose $f(x)$ is convex (but not necessarily strongly convex), then the sequence $\{(x^k,\lambda^k)\}$ generated by dual-primal balanced ALM \eqref{311} satisfies that for $x\in\mathbb{R}^n$ and $\lambda\in\mathbb{R}^m$,
	$$ \begin{aligned}
	&f(\hat{x}^{K})+\lambda^T(A\hat{x}^{K}-b)-f(x)-({\hat{\lambda}}^K)^T(Ax-b)\\
	\leq&\frac{1}{K+1}\Big(\frac{r}{2}\|x^0-x\|^2+\frac{1}{2r}( \|\lambda^{0}-\lambda\|_H^2+\|\lambda^{-1}-\lambda^0\|_H^2)-(\lambda^{-1}-\lambda^0)^TA(x^{0}-x)\Big),
	\end{aligned}
	$$
	where $\hat{x}^K= (\sum_{k=0}^{K}x^{k+1})/(K+1)$ and $\hat{\lambda}^K= (\sum_{k=0}^{K}\lambda^{k+1})/(K+1)$.
\end{thm}
\begin{proof}
With the setting $r^k=r$ and  $\delta= \delta'/r$,
Algorithm  \ref{Aa:2} reduces to Algorithm \ref{311}. Then,  according to Lemma \ref{al2}, we have
\begin{eqnarray} &&f(x^{k+1})+\lambda^T(Ax^{k+1}-b)-f(x)-({\lambda^{k+1}})^T(Ax-b)\nonumber\\
&\leq&	\frac{r}{2}\Big(\|x^k-x\|^2-\|x^{k+1}-x\|^2-\|x^k-x^{k+1}\|^2\Big)+\frac{1}{2r}\Big( \|\lambda^{k}-\lambda\|_H^2\nonumber\\
&&-\|\lambda^{k+1}-\lambda\|^2_H-\|\lambda^{k}-\lambda^{k+1}\|^2_H	\Big)+(2\lambda^k-\lambda^{k-1}-\lambda^{k+1})^TA(x^{k+1}-x),\nonumber\\
&\leq& 	\frac{r}{2}\Big(\|x^k-x\|^2-\|x^{k+1}-x\|^2\Big)\nonumber\\
&&+\frac{1}{2r}\Big( \|\lambda^{k}-\lambda\|_H^2-\|\lambda^{k+1}-\lambda\|^2_H+\|\lambda^{k-1}-\lambda^{k}\|^2_{AA^T}-\|\lambda^{k}-\lambda^{k+1}\|^2_{AA^T}	\Big)\nonumber\\
&&+(\lambda^k-\lambda^{k+1})^TA(x^{k+1}-x)-(\lambda^{k-1}-\lambda^k)^TA(x^{k}-x),
\label{lbd:8}
\end{eqnarray}
where the last inequality follows from
$$\begin{aligned}	&(2\lambda^k-\lambda^{k-1}-\lambda^{k+1})^TA(x^{k+1}-x)\\=&(\lambda^k-\lambda^{k+1})^TA(x^{k+1}-x)-(\lambda^{k-1}-\lambda^k)^TA(x^{k}-x)\\&+(\lambda^{k-1}-\lambda^{k})^TA(x^{k}-x^{k+1})\\\leq&(\lambda^k-\lambda^{k+1})^TA(x^{k+1}-x)-(\lambda^{k-1}-\lambda^k)^TA(x^{k}-x)\\
&+\frac{1}{2r}\|\lambda^{k-1}-\lambda^k\|^2_{AA^T}+\frac{r}{2}\|x^k-x^{k+1}\|^2.
	\end{aligned}
$$
Adding all the inequalities \eqref{lbd:8} from $k=0$ to $k=K$ and then dividing both sides by $(K+1)$ completes the proof.
\end{proof}

Different from the $O(1/K)$ convergence rate established in \cite{xu2021dual} based on the convex combination of $\{(x^{k},\lambda^{k+1})\}$,  our analysis relies on the new convex combination of iteration $\{(x^{k+1},\lambda^{k+1})\}$,  which can provide an $O(1/K^2)$ convergence rate of the accelerated dual-primal balanced ALM (Algorithm \ref{Aa:2}).

\begin{thm}\label{thma2}
Suppose $f(x)$ is $\mu$-strongly convex. For any $x\in\mathbb{R}^n$ and $\lambda\in\mathbb{R}^m$, the sequence $\{(x^k,\lambda^k)\}$ generated by Algorithm \ref{Aa:2} with the special setting $\theta^k= r^k/r^{k+1}$, $(r^k+\mu)r^k\geq (r^{k+1})^2$,  and $0<r^k\leq r^{k+1}$ satisfies that
	$$\begin{aligned}
	&\Big(\sum_{k=0}^{K}r^k\Big)\Big(f(\hat{x}^{K})+\lambda^T(A\hat{x}^{K}-b)-f(x)-({{\hat{\lambda}}^{K}})^T(Ax-b)\Big)\\\leq&
	\frac{(r^0)^2}{2}\|x^0-x\|^2+\frac{1}{2}\|\lambda^0-\lambda\|^2_H+\frac{1}{2}\|\lambda^{-1}-\lambda^0\|^2_H-r^{-1}(\lambda^{-1}-\lambda^0)^TA(x^{0}-x),
	\end{aligned}
	$$
	where $\hat{x}^K= (\sum_{k=0}^{K}r^kx^{k+1})/(\sum_{k=0}^{K}r^k)$ and $\hat{\lambda}^K=(\sum_{k=0}^{K}r^k\lambda^{k+1})/
(\sum_{k=0}^{K}r^k)$. In particular, with the setting $r^k= \mu(k+1)/3$, we have
$$f(\hat{x}^{K})+\lambda^T(A\hat{x}^{K}-b)-f(x)-({{\hat{\lambda}}^{K}})^T(Ax-b)\le O(1/K^2).$$
\end{thm}
\begin{proof}
First we can verify that
\begin{eqnarray}	&&(\widetilde{\lambda}^k-\lambda^{k+1})^TA(x^{k+1}-x)\nonumber\\
&=&(\lambda^k-\lambda^{k+1})^TA(x^{k+1}-x)-\theta^{k-1}(\lambda^{k-1}-\lambda^k)^TA(x^{k}-x)\nonumber\\
&&+\theta^{k-1}(\lambda^{k-1}-\lambda^{k})^TA(x^{k}-x^{k+1})\nonumber\\
&\leq&(\lambda^k-\lambda^{k+1})^TA(x^{k+1}-x)-\theta^{k-1}(\lambda^{k-1}-\lambda^k)^TA(x^{k}-x)\nonumber\\
&+&\frac{1}{2r^k}\|\lambda^{k-1}-\lambda^k\|^2_{AA^T}+\frac{(r^{k-1})^2}{2r^k}\|x^k-x^{k+1}\|^2\nonumber\\
&\leq&(\lambda^k-\lambda^{k+1})^TA(x^{k+1}-x)-\theta^{k-1}(\lambda^{k-1}-\lambda^k)^TA(x^{k}-x)\nonumber\\
&&+\frac{1}{2r^k}\|\lambda^{k-1}-\lambda^k\|^2_{H}+\frac{r^{k}}{2}\|x^k-x^{k+1}\|^2.
\label{acc:8}
	\end{eqnarray}

Multiplying  both sides of \eqref{acc2} by $r^k$ and noting
the fact $(r^k+\mu)r^k\geq (r^{k+1})^2$ and	\eqref{acc:8} yields that
\begin{eqnarray} &&r^k\Big(f(x^{k+1})+\lambda^T(Ax^{k+1}-b)-f(x)-({\lambda^{k+1}})^T(Ax-b)\Big)\nonumber\\
&\leq&	\frac{(r^k)^2}{2}\|x^k-x\|^2-\frac{(r^{k+1})^2}{2}\|x^{k+1}-x\|^2\nonumber\\
&&+\frac{1}{2}\Big( \|\lambda^{k}-\lambda\|_H^2-\|\lambda^{k+1}-\lambda\|^2_H +\|\lambda^{k-1}-\lambda^k\|^2_H-\|\lambda^k-\lambda^{k+1}\|^2_H\Big)\nonumber\\
&&-r^{k-1}(\lambda^{k-1}-\lambda^k)^TA(x^{k}-x)+r^k(\lambda^k-\lambda^{k+1})^TA(x^{k+1}-x) .\label{acc:9}
\end{eqnarray}
By adding all the inequalities \eqref{acc:9} from $k=0$ to $k=K$, we obtain
	$$ \begin{aligned}
	&\Big(\sum_{k=0}^{K}r^k\Big)\Big(f(\hat{x}^{K})+\lambda^T(A\hat{x}^{K}-b)-f(x)-({{\hat{\lambda}}^{K}})^T(Ax-b)\Big)\\\leq&
	\frac{(r^0)^2}{2}\|x^0-x\|^2-\frac{(r^{K+1})^2}{2}\|x^{K+1}-x\|^2\\&+\frac{1}{2}\Big( \|\lambda^{0}-\lambda\|_H^2-\|\lambda^{K+1}-\lambda\|^2_H
	+\|\lambda^{-1}-\lambda^0\|^2_H-\|\lambda^K-\lambda^{K+1}\|^2_H\Big)\\&-r^{-1}(\lambda^{-1}-\lambda^0)^TA(x^{0}-x)+r^K(\lambda^K-\lambda^{K+1})^TA(x^{K+1}-x).
	\end{aligned}
	$$
Then, by further observing that $$
	r^K(\lambda^K-\lambda^{K+1})^TA(x^{K+1}-x)\leq\frac{1}{2}\|\lambda^K-\lambda^{K+1}\|_H^2+\frac{(r^{K+1})^2}{2}\|x^{K+1}-x\|^2,
	$$we can complete the proof.
\end{proof}

\begin{rem}
As shown in  Remark \ref{re2}, Algorithm \ref{Aa:1}  is equivalent to proximal ADMM for solving the dual problem \eqref{dual22}. Analogously, Algorithm \ref{Aa:1} is expected to be equivalent to  the following proximal ADMM for solving \eqref{dual22}.
	\begin{eqnarray}\label{APADMM}
		\begin{cases}
			u^{k+1}=\arg\min\limits_{u}\{f^*(u)+(\bar{x}^k)^Tu+\frac{1}{2r^k}\|u+A^T\lambda^k\|^2\},\\
			\lambda^{k+1}=\arg\min\limits_{\lambda}\{b^T\lambda+(A\bar{x}^k)^T\lambda+\frac{1}{2r^k}\|u^{k+1}+A^T\lambda\|^2+\frac{\delta^k}{2}\|\lambda-\lambda^k\|^2\},\\
			\bar{x}^{k+1}=\bar{x}^k+\frac{1}{r^k}(u^{k+1}+A^T\lambda^{k+1}).
		\end{cases}
	\end{eqnarray}
	According to the optimality condition of $u$-subproblem, we have
	$$\begin{aligned}
	0\in& \partial f^*(u^{k+1})+\bar{x}^k+\frac{1}{r^k}(u^{k+1}+A^T\lambda^{k+1})\\=&
	\partial f^*(u^{k+1})+\bar{x}^{k+1}+\frac{1}{r^k}A^T(\lambda^k-\lambda^{k+1}).
	\end{aligned}$$
	Let $x^{k+1}=-\bar{x}^{k+1}-\frac{1}{r^k}A^T(\lambda^k-\lambda^{k+1})$. We obtain
	$$x^{k+1}\in\partial f^*(u^{k+1})\Longleftrightarrow u^{k+1}\in\partial f(x^{k+1}).
	$$Then, it holds that
	 $$\begin{aligned}
	0\in&\partial f(x^{k+1})-u^{k+1}=
	\partial f(x^{k+1})+r^k(\bar{x}^k-\bar{x}^{k+1})+A^T\lambda^{k+1}\\=&
	\partial f(x^{k+1})+A^T(\lambda^k+\frac{r^k}{r^{k-1}}(\lambda^k-\lambda^{k-1}))+r^k(x^{k+1}-x^k).
	\end{aligned}
	$$
	Accoding to the optimality condition of $\lambda$-subproblem, we have $$\begin{aligned}
	0=&b+A\bar{x}^k+\frac{1}{r^k}A(u^{k+1}+A^T\lambda^k)+\delta^k(\lambda^{k+1}-\lambda^k)\\=&
	b-Ax^{k+1}+(\frac{1}{r^k}AA^T+\delta^kI)(\lambda^{k+1}-\lambda^k).
	\end{aligned}
	$$
Hence, Algorithm \eqref{APADMM} is equivalent to
\begin{eqnarray}\label{GDBALM1}
\begin{cases}			
x^{k+1}=\mathop{\text{argmin}}
\{f(x)+\frac{r^k}{2}\|x-(x^k-\frac{1}{r^k}A^T(\lambda^k+\frac{r^k}{r^{k-1}}(\lambda^k-\lambda^{k-1}))\|^2\},\\
\lambda^{k+1}=\lambda^{k}+(\frac{1}{r^{k}}AA^T+\delta^k I_m)^{-1}[Ax^{k+1}-b].
\end{cases}
\end{eqnarray}
Algorithm \eqref{GDBALM1} is different from Algorithm \eqref{Aa:2}.
The $O(1/K^2)$ convergence rate established in Theorem \ref{thma2} cannot be extended for Algorithm \eqref{GDBALM1}.
To the best of our knowledge, it is unknown whether Algorithm \eqref{APADMM} enjoys an $O(1/K^2)$  convergence rate if $f(x)$ is $\mu$-strongly convex.
\end{rem}

\section{Conclusions}
We have proposed a lift-and-permute scheme of ADMM for solving convex programming problems with linear equality constraints.  We show that not only the recent balanced augmented Lagrangian method and its dual-primal variation, but also the proximal ADMM and Douglas-Rachford splitting algorithm correspond to special algorithms in our scheme. As extensions, we  propose accelerated algorithms with worst-case $O(1/k^2)$ convergence rates in the case that $f(x)$ is strongly convex. Our results can be easily generalized to solve more general convex programming problems with additional linear   inequality constraints.

We notice that each algorithm in our scheme has a fixed permuted order to update variables. It is interesting to consider algorithms with randomized order for updating variables in each iteration.
Future works also include  applying
our lift-and-permute scheme to the primal lifted problem and then studying their convergence rates and acceleration.
%For future works, we consider the following multi-block convex optimization problem:
%\begin{eqnarray}\label{4.7}
%%\begin{split}
%\min\limits_{x_i} \left\{\sum_{i=1}^{p}f_i(x_i):~
%\sum_{i=1}^{p}A_ix_i=b\right\},
%%\end{split}
%\end{eqnarray}
%where
%$f_i(x):\mathbb{R}^{n_i}\rightarrow\mathbb{R}$ is closed, proper, convex, but not necessarily smooth,  $A_i\in\mathbb{R}^{m\times n_i}$, $b\in\mathbb{R}^{m}$ are input data and $p\geq2$. Problem \eqref{4.7} can be rewritten as
%\begin{eqnarray}\label{4.8}
%%\begin{split}
%\min\limits_{x_i} \left\{F(x)=\sum_{i=1}^{p}f_i(x_i):~
%Ax=b\right\},
%%\end{split}
%\end{eqnarray}
%where $A=(A_1,...,A_p)$ and $x=(x_1^T,...,x_p^T)^T$. If the accelerated algoriths given in section 4 for solving  \eqref{4.8} means that  $f_i(x_i)$ for $\in\{1,...,p\}$ needs to be strongly convex, the future work is to bulid accelerated algorithm that only require one function to be strongly convex with solving $x_i$ in a Gauss-Seidel way.

\end{document}